\documentclass[11pt]{article}

\usepackage[sort]{natbib}
\usepackage[bookmarks=true,bookmarksnumbered=true,colorlinks=true,allcolors=black]{hyperref}

\usepackage{graphicx}
\usepackage{amsmath,amssymb,amsfonts,amsthm}
\renewcommand*{\baselinestretch}{1.25}

\usepackage{geometry}
\usepackage{enumitem}

\newtheorem{theorem}{Theorem}[section]
\newtheorem{lem}{Lemma}[section]
\newtheorem{prop}{Proposition}[section]
\newtheorem{cor}{Corollary}[section]
\theoremstyle{definition}

\newtheorem*{rmk*}{Remark}
\newtheorem{rmk}{Remark}[section]

\DeclareMathOperator{\trace}{tr}
\DeclareMathOperator{\diag}{diag}
\DeclareMathOperator{\spr}{spr}
\DeclareMathOperator{\variance}{Var}
\DeclareMathOperator{\argmax}{argmax}

\numberwithin{equation}{section}

\makeatletter
    \renewcommand*{\section}{\@startsection{section}{1}{\z@}%
    {10pt}{5pt}{\reset@font\normalsize\bfseries}}
\makeatother
    
\makeatletter
    \renewcommand*{\subsection}{\@startsection{subsection}{2}{\z@}%
    {5pt}{5pt}{\reset@font\normalsize\mdseries\itshape}}
\makeatother

\makeatletter
    \renewcommand*{\subsubsection}{\@startsection{subsubsection}{3}{\z@}%
    {5pt}{5pt}{\reset@font\normalsize\mdseries\itshape}}
\makeatother

\makeatletter
\def\@listi{\leftmargin\leftmargini
  \topsep=.5\baselineskip 
  \partopsep=0pt \parsep=0pt \itemsep=0pt}
\let\@listI\@listi
\@listi
\def\@listii{\leftmargin\leftmarginii
  \labelwidth\leftmarginii \advance\labelwidth-\labelsep
  \topsep=0pt \partopsep=0pt \parsep=0pt \itemsep=0pt}
\def\@listiii{\leftmargin\leftmarginiii
  \labelwidth\leftmarginiii \advance\labelwidth-\labelsep
  \topsep=0pt \partopsep=0pt \parsep=0pt \itemsep=0pt}
\def\@listiv{\leftmargin\leftmarginiv
  \labelwidth\leftmarginiv \advance\labelwidth-\labelsep
  \topsep=0pt \partopsep=0pt \parsep=0pt \itemsep=0pt}
\makeatother

\usepackage[pagewise,mathlines]{lineno}

\newcommand*\patchAmsMathEnvironmentForLineno[1]{%
  \expandafter\let\csname old#1\expandafter\endcsname\csname #1\endcsname
  \expandafter\let\csname oldend#1\expandafter\endcsname\csname end#1\endcsname
  \renewenvironment{#1}%
     {\linenomath\csname old#1\endcsname}%
     {\csname oldend#1\endcsname\endlinenomath}}%
\newcommand*\patchBothAmsMathEnvironmentsForLineno[1]{%
  \patchAmsMathEnvironmentForLineno{#1}%
  \patchAmsMathEnvironmentForLineno{#1*}}%
\AtBeginDocument{%
\patchBothAmsMathEnvironmentsForLineno{equation}%
\patchBothAmsMathEnvironmentsForLineno{align}%
\patchBothAmsMathEnvironmentsForLineno{flalign}%
\patchBothAmsMathEnvironmentsForLineno{alignat}%
\patchBothAmsMathEnvironmentsForLineno{gather}%
\patchBothAmsMathEnvironmentsForLineno{multline}%
}

\allowdisplaybreaks

\title{On the asymptotic structure of Brownian motions with\\
a small lead-lag effect}
\author{Yuta Koike\thanks{Department of Business Administration, Graduate School of Social Sciences, Tokyo Metropolitan University, Marunouchi Eiraku Bldg. 18F, 1-4-1 Marunouchi, Chiyoda-ku, Tokyo 100-0005 Japan}
\thanks{The Institute of Statistical Mathematics, 10-3 Midori-cho, Tachikawa, Tokyo 190-8562, Japan}
\thanks{CREST, Japan Science and Technology Agency}
}

\begin{document}


\maketitle

\begin{abstract}

This paper considers two Brownian motions in a situation where one is correlated to the other with a slight delay. 
We study the problem of estimating the time lag parameter between these Brownian motions from their high-frequency observations, which are possibly subject to measurement errors. The measurement errors are assumed to be i.i.d., centered Gaussian and independent of the latent processes. 
We investigate the asymptotic structure of the likelihood ratio process for this model when the lag parameter is asymptotically infinitesimal. 
We show that the structure of the limit experiment depends on the level of the measurement errors: If the measurement errors locally dominate the latent Brownian motions, the model enjoys the LAN property. Otherwise, the limit experiment does not result in typical ones appearing in the literature. 
We also discuss the efficient estimation of the lag parameter to highlight the statistical implications.
\vspace{3mm}

\noindent \textit{Keywords and phrases}: Asymptotic efficiency; Endogenous noise; Lead-lag effect; Local asymptotic normality; Microstructure noise.

\end{abstract}

\section{Introduction}

Let $B_t=(B^1_t,B^2_t)$ $(t\in\mathbb{R})$ be a bivariate two-sided Brownian motion such that $B_0=0$, $E[(B^1_1)^2]=E[(B^2_1)^2]=1$ and $E[B^1_1B^2_1]=\rho$ for some $\rho\in(-1,0)\cup(0,1)$. Also, let $\epsilon_i=(\epsilon^1_i,\epsilon^2_i)$ ($i=1,2,\dots$) be a sequence of i.i.d.~bivariate standard normal variables independent of $B$. For each $n\in\mathbb{N}$, we denote by $\mathbb{P}_{n,\vartheta}$ the law of the vector $\mathbf{Z}_n=(X_1,\dots,X_n,Y_1,\dots,Y_n)^\top$ generated by the following model:
\begin{equation}\label{model}
\left\{\begin{array}{lll}
X_i=B^1_{i/n}+\sqrt{v_n}\epsilon^1_i,&Y_i=B^2_{i/n-\vartheta}+\sqrt{v_n}\epsilon^2_i&\textrm{if }\vartheta\geq0,\\
X_i=B^1_{i/n-\vartheta}+\sqrt{v_n}\epsilon^1_i,&Y_i=B^2_{i/n}+\sqrt{v_n}\epsilon^2_i&\textrm{if }\vartheta<0
\end{array}\right.
\qquad\text{for }i=1,\dots,n,
\end{equation}
where $v_n$ is a non-negative number. $\vartheta\in\mathbb{R}$ denotes the unknown time-lag parameter which we are interested in. Especially, the sign of $\vartheta$ is unknown. The aim of this paper is to study the asymptotic structure of the sequence of experiments $(\mathbb{R}^{2n},\mathcal{B}^{2n},(\mathbb{P}_{n,\vartheta})_{\vartheta\in\mathbb{R}})$, $n=1,2,\dots$, as $n\to\infty$ when the time lag parameter $\vartheta$ is asymptotically infinitesimal, i.e.~$\vartheta$ tends to 0 as $n\to\infty$ (here and below $\mathcal{B}^m$ denotes the Borel $\sigma$-field of $\mathbb{R}^m$ for $m\in\mathbb{N}$). More precisely, we study the limit experiment of $(\mathbb{P}_{n,r_nu})_{u\in\mathbb{R}}$ as $n\to\infty$ for the proper convergence rate $r_n$.

If $v_n\equiv0$, model \eqref{model} is a special case of the Hoffmann-Rosenbaum-Yoshida (HRY) model introduced in \citet{HRY2013} to describe lead-lag effects in high-frequency financial data. A similar model has also been studied in \citet{RR2010} with an asymptotic regime different from the current setting. Here, the lead-lag effect refers to a situation where one time series is correlated to another time series at a later time, which has especially drawn attention in analysis of economic time series data for a long time, and associated econometric methods have been developed by many authors; see Section 1 of \cite{HRY2013}, Section 3 of \cite{RR2010} and references therein.   
The practicality of the HRY model in empirical work has recently been established by several authors such as \citet{AM2014}, \citet{HA2014} and \citet{BOW2017} for financial data and \citet{IPSS2015} for social media data. 
These empirical studies show that time lag parameters are typically comparable to the observation frequencies in their scales. This motivates us to study the HRY model when $\vartheta$ is small. In such a situation, one would especially be interested in how small lag parameters can be identified in principle. To the author's knowledge, however, there is few theoretical study for the HRY model and, in particular, nothing has been known about the optimality of statistical inferences for the HRY model. The purpose of this paper is trying to fill in this gap.   

In this paper, as well as (a special case of) the HRY model, we also consider a situation where the model contains measurement errors. This is motivated by the recent studies for the volatility estimation from ultra high frequency financial data, which is typically modeled as a discretely observed semimartingale with market microstructure noise. We refer to Chapter 7 of \citet{AJ2014} for a brief description of this subject. In particular, the asymptotic structure and the asymptotic efficiency bound have been established in the work of \citet{GJ2001a,GJ2001b} (see also \citet{CMSH2010}) for a statistical model of estimating the scale parameter $\sigma>0$ from the discrete observations
\begin{equation}\label{GJmodel}
\sigma W_{i/n}+\sqrt{v_n}\delta_i,\qquad i=1,\dots,n,
\end{equation}
where $W=(W_t)_{t\in[0,1]}$ is a one-dimensional standard Wiener process and $(\delta_i)_{i=1}^n$ is a sequence of centered i.i.d.~standard normal variables independent of $W$. They proved the LAN property for the above model and constructed asymptotically efficient estimators for $\sigma$ (they indeed considered a more general setting). Extensions of their LAN result to a multivariate setting have also been studied by several authors. The correlation estimation in a bivariate setting is studied in \cite{B2011a}, while a more general setting containing non-synchronous sampling case is studied in \cite{Ogihara2014noise}. On the other hand, \citet{Reiss2011} has studied the asymptotic structure of model \eqref{GJmodel} when $\sigma$ is a function of time rather than a constant and established the asymptotic equivalence between such a model and a Gaussian white noise model. The result has been extended to the bivariate case by \cite{BR2014} and a multivariate setting containing non-synchronous case by \cite{BHMR2014}. Another type of extension, replacing the Wiener process $W$ by a different continuous-time process, has also been studied. For example, \citet{SSH2014} consider the efficient estimation of $\sigma$ in a situation where $W$ is a more general Gaussian process, especially a fractional Brownian motion.  

The main contribution of this paper is (i) to determine the proper convergence rate $r_n$, and (ii) to derive a stochastic expansion for the likelihood ratio process for $(\mathbb{P}_{n,r_nu})_{u\in\mathbb{R}}$. 
Analogously to \cite{GJ2001a}, the proper convergence rate $r_n$ depends on the behavior of the sequence $nv_n$. This is intuitively natural because $\variance[B_{i/n}-B_{(i-1)/n}]=n^{-1}$ and thus the behavior of $nv_n$ determines how strongly the measurement errors (locally) dominate the nature of the observed returns. In particular, we find that $r_n=n^{-\frac{3}{2}}$ if $v_n\equiv0$, or more generally if $nv_n$ is bounded.\footnote{Indeed, an intuition for this fact has already been appeared in \cite{HRY2013} (see Remark \ref{intuition}).} The rate $n^{-\frac{3}{2}}$ is much faster than the usual parametric rate $n^{-\frac{1}{2}}$ and even faster than the rate $n^{-1}$. Since the time resolution of our model is $n^{-1}$, our result suggests that we could estimate lag parameters smaller than the time resolution of observation data. This implication is at least true for our restrictive situation, as shown in Section \ref{section:application}. 
Since the convergence rate of the estimator for the lag parameter $\vartheta$ proposed in \cite{HRY2013} cannot be faster than $n^{-1}$ (see Proposition 2 of \cite{HRY2013} and the discussion after this proposition), our result shows that their estimator is suboptimal in the setting considered in this paper (although their estimator works in a more general setting).

Given the proper convergence rate, we have the following stochastic expansion for the likelihood ratio process: There are random variables $\mathsf{T}_n$ and $\mathsf{S}_n$ defined on $(\mathbb{R}^{2n},\mathcal{B}^{2n})$ and non-negative numbers $I_\gamma$ and $J_\gamma$ such that
\begin{equation}\label{expansion}
\log\frac{\mathrm{d}\mathbb{P}_{n,r_nu_n}}{\mathrm{d}\mathbb{P}_{n,0}}-\left\{
u_n\mathsf{T}_n+|u_n|\mathsf{S}_n-\frac{u_n^2}{2}(I_\gamma+J_\gamma)\right\}\xrightarrow{p}0\qquad
\text{under }~\mathbb{P}_{n,0}\quad\text{as $n\to\infty$}
\end{equation}
for any bounded sequence $u_n$ of real numbers and
\begin{equation}\label{CLT}
(\mathsf{T}_n,\mathsf{S}_n)\xrightarrow{d} \mathcal{N}(0,I_\gamma)\otimes \mathcal{N}(0,J_\gamma)\qquad
\text{under }~\mathbb{P}_{n,0}\quad\text{as $n\to\infty$}.
\end{equation}    
Therefore, by a contiguity argument we deduce that the experiments $(\mathbb{R}^{2n},\mathcal{B}^{2n},(\mathbb{P}_{n,r_nu})_{u\in\mathbb{R}})$ converge weakly to the experiment $(\mathbb{R}^2,\mathcal{B}^2,(\mathbb{Q}_u)_{u\in\mathbb{R}})$ in the Le Cam sense, where $\mathbb{Q}_u=\mathcal{N}(uI_\gamma,I_\gamma)\otimes\mathcal{N}(|u|J_\gamma,J_\gamma)$ (see Corollary \ref{limit}). 
The numbers $I_\gamma$ and $J_\gamma$ are determined by the asymptotic behavior of $nv_n$ and precisely defined by \eqref{def:I}--\eqref{def:J}. In particular, $I_\gamma$ is always positive, while $J_\gamma$ is positive if $nv_n$ is bounded and $J_\gamma=0$ otherwise. The case $J_\gamma=0$ corresponds to the situation where the measurement errors locally dominate the signal, and in this case our model enjoys the LAN property which commonly appears in regular experiments. This result is of interest because model \eqref{model} exhibits irregularity in the sense that its likelihood function is not smooth in $\vartheta$, and the limit experiment of such a model typically deviates from the LAN structure as illustrated in Chapters V--VII of \citet{IH1981}. Our result means that the measurement errors have a kind of regularizing effect on the asymptotic structure of model \eqref{model}. 
On the other hand, if $J_\gamma>0$, which corresponds to the cases where the signal dominates or is balanced with the measurement errors, in addition to an observation from a usual Gaussian shift experiment $ \mathcal{N}(u,I_\gamma^{-1})$, the limit experiment contains an extra observation from the experiment $ \mathcal{N}(|u|,J_\gamma^{-1})$. 
Although this experiment looks simple, to the author's knowledge it does not result in well-studied cases (such as in \cite{IH1981}), so the definition of asymptotically efficient estimators in this case is not obvious. To obtain the asymptotic efficiency bound for estimating the lag parameter in this case, in Section \ref{section:application} we apply \citet{IH1981}'s theory to our problem, which is a common approach to establish asymptotic efficiency bounds for experiments generated by diffusion type processes (see \citet{Kutoyants2004} for details). 
As a result, we find that Bayesian estimators are asymptotically efficient, while the maximum likelihood estimator is not always asymptotically efficient. This is a common phenomenon in irregular models; see Chapters V--VII of \cite{IH1981}, \cite{KK2000}, Chapter 3 of \cite{Kutoyants2004}, \cite{RS1995} and Chapter 9 of \cite{Vaart1998} for example. 


This paper is organized as follows. Section \ref{section:main} presents the main result of the paper. In Section \ref{section:application} we discuss the efficient estimation of the lag parameter in our setting. Section \ref{section:technical} is devoted to the proof of an auxiliary technical result.

\subsection*{General notation}

$E_m$ denotes the $m\times m$-identity matrix. For a  matrix $A$, we denote by $\|A\|_\mathrm{sp}$ and $\|A\|_F$ its spectral norm and the Frobenius norm, respectively. That is, $\|A\|_\mathrm{sp}=\sup\{\|Ax\|:\|x\|\leq 1\}$ and $\|A\|_F^2=\trace(A^\top A)$. Also, we denote by $A^{ij}$ the $(i,j)$-th entry of $A$.

\section{Main result}\label{section:main}

We start with completing the definitions of the quantities $r_n$, $I_\gamma$ and $J_\gamma$ appearing in the Introduction. First, following \cite{GJ2001a} we assume that the sequence $nv_n$ converges in $[0,\infty]$ and set
\[
\gamma:=\lim_{n\to\infty}nv_n.
\]
We also assume $\limsup_nv_n<\infty$ as in \cite{GJ2001a}. Then we set
\[
N_n=
\left\{
\begin{array}{cl}
\sqrt{n/v_n}&\textrm{if }\gamma=\infty,\\
n  & \textrm{otherwise}.      
\end{array}
\right.
\]
$N_n$ can be considered as an ``effective'' sample size in the sense that the proper convergence rate for estimating $\sigma$ from model \eqref{model} is given by $N_n^{-\frac{1}{2}}$, which is seen as the regular parametric rate if we regard $N_n$ as the sample size. Using this effective sample size $N_n$, we define our proper convergence rate as $r_n=N_n^{-\frac{3}{2}}$. The constants $I_\gamma$ and $J_\gamma$ appearing in \eqref{expansion}--\eqref{CLT} are defined by
\begin{equation}\label{def:I}
I_\gamma
=\left\{
\begin{array}{ll}
\frac{\rho^2}{2(1-\rho^2)}  & \mathrm{if}~\gamma=0,  \\
\frac{\rho\left(\sqrt{(1+\rho)(1+\rho+4\gamma)}-\sqrt{(1-\rho)(1-\rho+4\gamma)}-2\rho\right)}{8\gamma^2}&\mathrm{if}~0<\gamma<\infty,\\
\frac{\rho^2}{2(\sqrt{1+\rho}+\sqrt{1-\rho})}  & \mathrm{if}~\gamma=\infty   
\end{array}
\right.
\end{equation}
and
\begin{equation}\label{def:J}
J_\gamma=\left\{\begin{array}{ll}
\frac{3}{4}\left\{\frac{1}{(1+\rho)^2}+\frac{1}{(1-\rho)^2}\right\}&\mathrm{if}~\gamma=0,\\
\frac{1}{16\gamma^2}\left[4-3\left\{\left(\frac{1+\rho}{1+\rho+4\gamma}\right)^{\frac{1}{2}}+\left(\frac{1-\rho}{1-\rho+4\gamma}\right)^{\frac{1}{2}}\right\}+\left\{\left(\frac{1+\rho}{1+\rho+4\gamma}\right)^{\frac{3}{2}}+\left(\frac{1-\rho}{1-\rho+4\gamma}\right)^{\frac{3}{2}}\right\}\right]&\mathrm{if}~0<\gamma<\infty,\\
0 &\mathrm{if}~\gamma=\infty.
\end{array}\right.
\end{equation}

\begin{rmk}
$I_\gamma$ is always positive for any $\gamma\in[0,\infty]$. This is evident when $\gamma=0$ or $\gamma=\infty$. When $0<\gamma<\infty$, this is proven as follows. First suppose that $0<\rho<1$. Then we have
\begin{align*}
&(1+\rho)(1+\rho+4\gamma)-\{\sqrt{(1-\rho)(1-\rho+4\gamma)}+2\rho\}^2\\
&=4\rho+8\gamma\rho-4\rho\sqrt{(1-\rho)^2+4\gamma(1-\rho)}-4\rho^2\\
&=4\rho\left((1-\rho)+2\gamma-\sqrt{(1-\rho)^2+4\gamma(1-\rho)}\right)\\
&=4\rho\left(\sqrt{(1-\rho)^2+4\gamma(1-\rho)+4\gamma^2}-\sqrt{(1-\rho)^2+4\gamma(1-\rho)}\right)
>0.
\end{align*}
Hence we have $I_\gamma>0$. On the other hand, if $-1<\rho<0$, applying the above inequality with replacing $\rho$ by $-\rho$, we obtain $\sqrt{(1-\rho)(1-\rho+4\gamma)}>\sqrt{(1+\rho)(1+\rho+4\gamma)}-2\rho$. Hence we have $I_\gamma>0$.
\end{rmk}

The following statement is our main result.
\begin{theorem}\label{thm:main}
There are two sequences $\mathsf{T}_n$ and $\mathsf{S}_n$ of random variables satisfying \eqref{expansion}--\eqref{CLT} for any bounded sequence $u_n$ of real numbers. 
\end{theorem}

We can explicitly give the variables $\mathsf{T}_n$ and $\mathsf{S}_n$ in Theorem \ref{thm:main} by \eqref{ts} below. Theorem \ref{thm:main} has some immediate consequences. The first one is the direct consequence of the definition of the LAN property.
\begin{cor}
If $\gamma=\infty$, $(\mathbb{P}_{n,\vartheta})_{\vartheta\in\mathbb{R}}$ has the LAN property at $\vartheta=0$ with rate $r_n$ and asymptotic Fisher information $I_\gamma$. 
\end{cor}

The second one follows from Le Cam's first lemma (see e.g.~Lemma 6.4 of \cite{Vaart1998}).
\begin{cor}\label{contiguity}
$\mathbb{P}_{n,\vartheta_n}$ and $\mathbb{P}_{n,0}$ are mutually contiguous if the sequence $\vartheta_n$ of real numbers satisfies $\vartheta_n=O(r_n)$.
\end{cor}

The third one is derived from Corollary \ref{contiguity} and Theorem 61.6 of \cite{Strasser1985} (we refer to \cite{DVW2009}, \cite{LeCam1986}, Chapter 10 of \cite{Strasser1985} and Chapters 8--9 of \cite{Vaart1998} for the definition and applications of the weak convergence of experiments).
\begin{cor}\label{limit}
The sequence $(\mathbb{R}^{2n},\mathcal{B}^{2n},(\mathbb{P}_{n,r_nh})_{h\in\mathbb{R}})$ of experiments converges weakly to the experiment $(\mathbb{R}^2,\mathcal{B}^2,$ $(\mathbb{Q}_u)_{u\in\mathbb{R}})$ as $n\to\infty$, where $\mathbb{Q}_u=\mathcal{N}(uI_\gamma,I_\gamma)\otimes\mathcal{N}(|u|J_\gamma,J_\gamma)$.
\end{cor}

Now we turn to the proof of Theorem \ref{thm:main}. Although $(\mathbb{P}_{n,\vartheta})_{\vartheta\in\mathbb{R}}$ consists of Gaussian distributions, the problem is not simple because the covariance matrix $C_n(\vartheta)$ of $\mathbb{P}_{n,\vartheta}$ is a complicated function of the lag parameter $\vartheta$. In particular, $C_n(\vartheta)$ and $C_n(\vartheta')$ are not simultaneously diagonalizable in general (even asymptotically) if $\vartheta\neq\vartheta'$. This could be troublesome because in analysis of Gaussian experiments the (asymptotically) simultaneous diagonalizability of the covariance matrices of the statistical model for different parameters typically plays an important role (cf.~Section 3 of \cite{Davies1973}, Lemma 8.1 of \cite{GJ2001a} and Lemma C.4 of \cite{SSH2014}). For this reason we first transfer from the model $\mathbb{P}_{n,\vartheta}$ to a more tractable model defined as follows: For each $n\in\mathbb{N}$, set $\Theta_n=\{\vartheta\in\mathbb{R}:v_n-n\vartheta^2+|\vartheta|\geq0\}=\{\vartheta\in\mathbb{R}:|\vartheta|\leq(1+\sqrt{1+4nv_n})/(2n)\}$. Then, for each $\vartheta\in\Theta_n$ we denote by $\widetilde{\mathbb{P}}_{n,\vartheta}$ the law of the vector $\widetilde{\mathbf{Z}}_n=(\widetilde{X}_1,\dots,\widetilde{X}_n,\widetilde{Y}_1,\dots,\widetilde{Y}_n)^\top$ defined by
\begin{equation}\label{KLmodel1}
\widetilde{X}_i=B^1_{i/n}+\widetilde{\epsilon}^1_i,\qquad
\widetilde{Y}_i=B^2_{i/n}+\widetilde{\epsilon}^2_i,
\end{equation}
where
\begin{equation}\label{KLmodel2}
\left\{\begin{array}{ll}
\widetilde{\epsilon}^1_{i,n}=\epsilon^1_i,\qquad\widetilde{\epsilon}^2_{i,n}=-n\vartheta(B^2_{i/n}-B^2_{(i-1)/n})+\sqrt{v_n-n\vartheta^2+\vartheta}\epsilon^2_i&\textrm{if }\vartheta\geq0,\\
\widetilde{\epsilon}^1_{i,n}=-n|\vartheta|(B^1_{i/n}-B^1_{(i-1)/n})+\sqrt{v_n-n\vartheta^2+|\vartheta|}\epsilon^1_i,\qquad\widetilde{\epsilon}^2_{i,n}=\epsilon^2_i&\textrm{if }\vartheta<0
\end{array}\right.
\end{equation}
for $i=1,\dots,n$. We denote by $\widetilde{C}_n(\vartheta)$ the covariance matrix of $\widetilde{\mathbb{P}}_{n,\vartheta}$.

In the following we will show that $\mathbb{P}_{n,\vartheta}$ is well-approximated by $\widetilde{\mathbb{P}}_{n,\vartheta}$. To be precise, the Hellinger distance between $\mathbb{P}_{n,\vartheta}$ and $\widetilde{\mathbb{P}}_{n,\vartheta}$ tends to 0 as $n\to\infty$, provided that $\vartheta$ tends to 0 sufficiently fast. Here, the Hellinger distance $H(P,Q)$ between two probability measures $P$ and $Q$ on a measurable space $(\mathcal{X},\mathcal{A})$ is defined by
\[
H(P,Q)=\left(\int_\mathcal{X}\left(\sqrt{\frac{\mathrm{d}P}{\mathrm{d}\mu}}-\sqrt{\frac{\mathrm{d}Q}{\mathrm{d}\mu}}\right)^2\mathrm{d}\mu\right)^{1/2},
\]
where $\mu$ is a $\sigma$-finite measure dominating both $P$ and $Q$ ($\mu=P+Q$ for example). It can easily be checked that $H(P,Q)$ does not depend on the choice of $\mu$. See Appendix A.1 of \cite{Reiss2011}, Section 2 of \cite{Strasser1985} and Section 2.4 of  \cite{Tsybakov2009} for more information about the Hellinger distance.

Throughout the paper, we denote by $\mathbb{E}_{n,\vartheta}$ (resp.~$\widetilde{\mathbb{E}}_{n,\vartheta}$) expectation with respect to $\mathbb{P}_{n,\vartheta}$ (resp.~$\widetilde{\mathbb{P}}_{n,\vartheta}$).
\begin{prop}\label{prop:hellinger}
\noindent{\normalfont(a)} If $|\vartheta|\leq1/n$, then $\mathbb{P}_{n,\vartheta}=\widetilde{\mathbb{P}}_{n,\vartheta}$. 

\noindent{\normalfont(b)} If $v_n>0$, $H^2(\mathbb{P}_{n,\vartheta},\widetilde{\mathbb{P}}_{n,\vartheta})\leq4v_n^{-2}(4+3\rho^2)n^2|\vartheta|^3$ for any $n\in\mathbb{N}$ and any $\vartheta\in\Theta_n$.

\noindent{\normalfont(c)} If a sequence $\vartheta_n$ of positive numbers satisfies $\vartheta_n=o(n^{-1}\vee N_n^{-\frac{4}{3}})$ as $n\to\infty$, then $\sup_{|\vartheta|\leq\vartheta_n}H(\mathbb{P}_{n,\vartheta},\widetilde{\mathbb{P}}_{n,\vartheta})\to0$.
\end{prop}

\begin{proof}
Claim (c) immediately follows from (a) and (b), so we focus on (a) and (b). By symmetry we may assume $\vartheta\geq0$. Let $x_1,\dots,x_n,y_1,\dots,y_n$ be the canonical variables on $\mathbb{R}^{2n}$. Then we have $\mathbb{E}_{n,\vartheta}[x_ix_j]=\widetilde{\mathbb{E}}_{n,\vartheta}[x_ix_j]$ for all $i,j$. Moreover, a simple computation shows that
\[
\mathbb{E}_{n,\vartheta}[y_iy_j]=
\left\{
\begin{array}{ll}
(\frac{i\wedge j}{n}-\vartheta)+v_n1_{\{i=j\}}&\textrm{if }i\wedge j\geq n\vartheta,\\
 (\vartheta-\frac{i\vee j}{n})_++v_n1_{\{i=j\}}&\textrm{otherwise}
\end{array}
\right.
\]
and
\begin{align*}
\widetilde{\mathbb{E}}_{n,\vartheta}[y_iy_j]&=\left(\frac{i\wedge j}{n}-\vartheta1_{\{i\geq j\}}-\vartheta1_{\{i\leq j\}}\right)+(v_n+\vartheta)1_{\{i=j\}}
=\left(\frac{i\wedge j}{n}-\vartheta\right)+v_n1_{\{i=j\}}
\end{align*}
and
\[
\mathbb{E}_{n,\vartheta}[x_iy_j]=
\left\{
\begin{array}{ll}
\rho i/n&\textrm{if }i< j-n\vartheta,\\
\rho(j/n-\vartheta)_+&\textrm{otherwise},
\end{array}
\right.
\qquad
\widetilde{\mathbb{E}}_{n,\vartheta}[x_iy_j]=
\left\{
\begin{array}{ll}
\rho i/n&\textrm{if }i< j,\\
\rho(j/n-\vartheta)&\textrm{otherwise}.
\end{array}
\right.
\]
Therefore, we have $C_n(\vartheta)=\widetilde{C}_n(\vartheta)$ if $\vartheta\leq1/n$. This yields claim (a) because both $\mathbb{P}_{n,\vartheta}$ and $\widetilde{\mathbb{P}}_{n,\vartheta}$ are centered Gaussian. On the other hand, from the above identities we also have
\begin{align*}
&\|C_n(\vartheta)-\widetilde{C}_n(\vartheta)\|_F^2\\
&=\sum_{\begin{subarray}{c}
i,j=1\\
i\wedge j<n\vartheta
\end{subarray}}^n\left|\left(\vartheta-\frac{i\vee j}{n}\right)_+-\left(\frac{i\wedge j}{n}-\vartheta\right)\right|^2\\
&\qquad+2\sum_{i=1}^n\left\{\sum_{j:i<j\leq i+n\vartheta}\left|\rho\left\{\left(\frac{j}{n}-\vartheta\right)_+-\frac{i}{n}\right\}\right|^2
+\sum_{j:j\leq i}\left|\rho\left\{\left(\frac{j}{n}-\vartheta\right)_+-\left(\frac{j}{n}-\vartheta\right)\right\}\right|^2\right\}\\
&=\sum_{\begin{subarray}{c}
i,j=1\\
i\wedge j<n\vartheta
\end{subarray}}^n\left|\left(\vartheta-\frac{i\vee j}{n}\right)_+-\left(\frac{i\wedge j}{n}-\vartheta\right)\right|^2\\
&\qquad+2\sum_{i=1}^n\left\{\sum_{j:i<j\leq n\vartheta}\left|\rho\frac{i}{n}\right|^2
+\sum_{j:i\vee n\vartheta<j\leq i+n\vartheta}\left|\rho\left(\frac{j-i}{n}-\vartheta\right)\right|^2
+\sum_{j:j\leq i\wedge n\vartheta}\left|\rho\left(\frac{j}{n}-\vartheta\right)\right|^2\right\}\\
&\leq 2n\cdot n\vartheta\cdot\left(\frac{2n\vartheta}{n}\right)^2+ 6n\cdot \rho^2\cdot n\vartheta\cdot\left(\frac{n\vartheta}{n}\right)^2
=(8+6\rho^2)n^2\vartheta^3.
\end{align*}
Now if $v_n>0$, $C_n(\vartheta)$ is positive semidefinite and satisfies $\|C_n(\vartheta)^{-\frac{1}{2}}\|_\mathrm{sp}\leq v_n^{-\frac{1}{2}}$ by the monotonicity theorem for eigenvalues (Corollary 4.3.3 of \cite{HJ1985}) because $C_n(\vartheta)-v_nE_{2n}$ is positive semidefinite. Therefore, by Eqs.(A.4) and (A.6) from \cite{Reiss2011} we obtain
\begin{align*}
H^2(\mathbb{P}_{n,\vartheta},\widetilde{\mathbb{P}}_{n,\vartheta})
\leq2\|C_n(\vartheta)^{-\frac{1}{2}}(C_n(\vartheta)-\widetilde{C}_n(\vartheta))C_n(\vartheta)^{-\frac{1}{2}}\|_F^2
\leq4v_n^{-2}(4+3\rho^2)n^2\vartheta^3,
\end{align*}
hence claim (b) holds true.
\end{proof}

In the following we will frequently use the fact that the Hellinger distance dominates the total variation distance:
\begin{equation}\label{eq:A1}
V(P,Q)\leq H(P,Q),
\end{equation}
where $V(P,Q)=\sup_{A\in\mathcal{A}}|P(A)-Q(A)|$. See e.g.~Lemma 2.3 of \cite{Tsybakov2009} for the proof. The following properties of the total variation distance are immediate consequences of the definition and important for our purpose. For each $n\in\mathbb{N}$, let $P_n$ and $Q_n$ be two sequences of probability measures on a measurable space $(\mathcal{X}_n,\mathcal{A}_n)$, and let $\zeta_n$ be a random variable on $(\mathcal{X}_n,\mathcal{A}_n)$ taking its value in a metric space $D$. Then, for any $a\in D$ and any probability measure $\mu$ on $D$, the following statements hold true:  
\begin{equation}\label{total}
\left.\begin{array}{lcl}
V(P_n,Q_n)\to0,~\zeta_n\xrightarrow{p}a~\text{ under }~P_n &\Rightarrow&\xi_n\xrightarrow{p}a~\text{ under }~Q_n,\\
V(P_n,Q_n)\to0,~\zeta_n\xrightarrow{d}\mu~\text{ under }~P_n &\Rightarrow&\xi_n\xrightarrow{d}\mu~\text{ under }~Q_n.
\end{array}\right\}
\end{equation}

Next we express the covariance matrix $\widetilde{C}_n(\vartheta)$ of $\widetilde{\mathbb{P}}_{n,\vartheta}$ as a tractable form. For this purpose we introduce some notation. The $n\times n$ matrix $\nabla_n$ denotes the backward difference operator, i.e.
\[
\nabla_n=
\left[
\begin{array}{cccc}
1  &   &  & \\
-1  &  1 &  &  \\
  &  \ddots & \ddots & \\
  &              & -1 & 1  
\end{array}
\right].
\]
We set $\widehat{\mathbf{Z}}_n:=(\nabla_n\oplus\nabla_n)\widetilde{\mathbf{Z}}_n=(\widetilde{X}_1,\widetilde{X}_2-\widetilde{X}_1,\dots,\widetilde{X}_n-\widetilde{X}_{n-1},\widetilde{Y}_1,\widetilde{Y}_2-\widetilde{Y}_1,\dots,\widetilde{Y}_n-\widetilde{Y}_{n-1})^\top$ and denote by $V_n(\vartheta)$ the covariance matrix of $\widehat{\mathbf{Z}}_n$, i.e.~$V_n(\vartheta)=(\nabla_n\oplus\nabla_n)\widetilde{C}_n(\vartheta)(\nabla_n\oplus\nabla_n)^\top$. $V_n(\vartheta)$ can explicitly be expressed as $V_n(\vartheta)=\bar{G}_n-\rho\vartheta\bar{\nabla}_n^{\mathrm{sign}(\vartheta)}$, where 
\[
\bar{G}_n=\left[
\begin{array}{cc}
G_n  & \frac{\rho}{n}E_n \\
\frac{\rho}{n}E_n & G_n
\end{array}
\right],\qquad
\bar{\nabla}_n^+=\left[
\begin{array}{cc}
0  & \nabla_n^\top   \\
\nabla_n  & \rho^{-1}R_n
\end{array}
\right],\qquad
\bar{\nabla}_n^-=-\left[
\begin{array}{cc}
\rho^{-1}R_n  & \nabla_n   \\
\nabla_n^\top  & 0
\end{array}
\right]
\]
with $G_n=\frac{1}{n}E_n+v_n F_n$, $F_n=\nabla_n\nabla_n^\top$ and
\[
R_n=
\left[
\begin{array}{cccc}
1  & 0  & \cdots & 0 \\
0  &  0 &  \cdots & 0 \\
\vdots  &  \ddots & \ddots & \vdots \\
0  &  \cdots   & 0 & 0 
\end{array}
\right].
\]
It is more convenient to rewrite the expression $\bar{\nabla}_n^{\mathrm{sign}(\vartheta)}$ as follows. Let $S_n$ and $T_n$ be the symmetric and skew-symmetric parts of $2\nabla_n^\top$, respectively. That is, $S_n=\nabla_n^\top+\nabla_n$ and $T_n=\nabla_n^\top-\nabla_n$. Then we set
\[
\bar{S}_n=\frac{1}{2}\left[
\begin{array}{cc}
\rho^{-1}R_n  & S_n   \\
S_n  & \rho^{-1}R_n
\end{array}
\right],\qquad
\bar{T}_n=\frac{1}{2}\left[
\begin{array}{cc}
-\rho^{-1}R_n  & T_n   \\
-T_n  & \rho^{-1}R_n
\end{array}
\right].
\]
We can easily check $\bar{\nabla}_n^\pm=\bar{T}_n\pm\bar{S}_n$, so we obtain $V_n(\vartheta)=\bar{G}_n-\rho(\vartheta\bar{T}_n+|\vartheta|\bar{S}_n)$. This is a simple function of $\vartheta$, so $V_n(\vartheta)$ is more tractable than $C_n(\vartheta)$: Although $V_n(\vartheta)$'s are not simultaneously diagonalizable for different $\vartheta$'s, it is sufficient to consider a relationship between the matrices $\bar{G}_n$, $\bar{T}_n$ and $\bar{S}_n$. In fact, it turns out that the following result is sufficient for our purpose. 
\begin{prop}\label{prop:main}
For any $\alpha,\beta\in\mathbb{R}$, we have $\|\bar{G}_n^{-\frac{1}{2}}(\alpha\bar{T}_n+\beta\bar{S}_n)\bar{G}_n^{-\frac{1}{2}}\|_\mathrm{sp}=O(N_n)$ as $n\to\infty$ and
\begin{equation}\label{information}
\lim_{n\to\infty}\rho^2r_n^2\|\bar{G}_n^{-\frac{1}{2}}(\alpha\bar{T}_n+\beta\bar{S}_n)\bar{G}_n^{-\frac{1}{2}}\|_F^2=2(\alpha^2I_\gamma+\beta^2J_\gamma).
\end{equation}
\end{prop}

The proof of Proposition \ref{prop:main} consists of elementary but complicated calculations, so we postpone it to the Appendix (Section \ref{section:technical}). 
We remark that the proof requires a calculation essentially different from that of the Fisher information for the scale parameter estimation from observations of the form \eqref{GJmodel}  such as in \cite{GJ2001a} and \cite{SSH2014} (see Remark \ref{fisher}). 
Also, note that Proposition \ref{prop:main} yields the invertibility of $V_n(\vartheta_n)$ for sufficiently large $n$ if $\vartheta_n=o(N_n^{-1})$ because $V_n(\vartheta_n)=\bar{G}_n^{\frac{1}{2}}(E_{2n}-\rho\bar{G}_n^{-\frac{1}{2}}(\vartheta_n\bar{T}_n+|\vartheta_n|\bar{S}_n)\bar{G}_n^{-\frac{1}{2}})\bar{G}_n^{\frac{1}{2}}$.

\begin{proof}[\upshape{\textbf{Proof of Theorem \ref{thm:main}}}]
Define the function $\widehat{z}_n:\mathbb{R}^{2n}\to\mathbb{R}^{2n}$ by setting $\widehat{z}_n(\zeta)=(\nabla_n\oplus\nabla_n)\zeta$ for $\zeta\in\mathbb{R}^{2n}$. Then we set
\begin{equation}\label{ts}
\mathsf{T}_n=-\frac{\rho}{2}r_n\left\{\widehat{z}_n^\top\bar{G}_n^{-1}\bar{T}_n\bar{G}_n^{-1}\widehat{z}_n-\trace(\bar{G}_n^{-1}\bar{T}_n)\right\},\qquad
\mathsf{S}_n=-\frac{\rho}{2}r_n\left\{\widehat{z}_n^\top\bar{G}_n^{-1}\bar{S}_n\bar{G}_n^{-1}\widehat{z}_n-\trace(\bar{G}_n^{-1}\bar{S}_n)\right\}.
\end{equation}
By virtue of Proposition \ref{prop:hellinger} and \eqref{eq:A1}--\eqref{total}, it suffices to prove the following statements:
\begin{align}
&\log\frac{\mathrm{d}\mathbb{P}_{n,r_nu_n}}{\mathrm{d}\mathbb{P}_{n,0}}
-\left\{u_n\mathsf{T}_n+|u_n|\mathsf{S}_n-\frac{u_n^2}{2}(I_\gamma+J_\gamma)\right\}\xrightarrow{p}0~\text{ under }~\widetilde{\mathbb{P}}_{n,0},\label{expansion1}\\
&(\mathsf{T}_n,\mathsf{S}_n)\xrightarrow{d}  \mathcal{N}(0,I_\gamma)\otimes   \mathcal{N}(0,J_\gamma)~\text{ under }~\widetilde{\mathbb{P}}_{n,0}.\label{normality}
\end{align}
\eqref{normality} follows from Proposition \ref{prop:main} and Proposition 2 of \cite{DY2011}. On the other hand, setting $A_n=\rho r_n\bar{G}_n^{-\frac{1}{2}}(u_n\bar{T}_n+|u_n|\bar{S}_n)\bar{G}_n^{-\frac{1}{2}}$, we have $\|A_n\|_F^2-2u_n^2(I_\gamma+J_\gamma)\to0$ by Proposition \ref{prop:main}. Therefore, by Proposition \ref{prop:hellinger}, \eqref{eq:A1} and Proposition 2 from Chapter 4 of \cite{LeCam1986} we obtain \eqref{expansion1} once we show that
\begin{equation}\label{expansion2}
\xi_n:=\log\frac{\mathrm{d}\widetilde{\mathbb{P}}_{n,r_nu_n}}{\mathrm{d}\widetilde{\mathbb{P}}_{n,0}}
-\left\{u_n\mathsf{T}_n+|u_n|\mathsf{S}_n-\frac{\|A_n\|_F^2}{4}\right\}\xrightarrow{p}0~\text{ under }~\widetilde{\mathbb{P}}_{n,0}.
\end{equation}

The strategy of the proof of \eqref{expansion2} is the same as that of Theorem 4.2 from \citet{Davies1973}. First, by Eq.(4.3) of \cite{Davies1973}, for sufficiently large $n$ we have
\begin{align*}
\widetilde{\mathbb{E}}_{n,0}[\xi_n]&=-\frac{1}{2}\left\{\log\det V_n(r_nu_n)-\log\det V_n(0)+\trace(V_n(0)(V_n(r_nu_n)^{-1}-V_n(0)^{-1}))\right\}+\frac{\|A_n\|_F^2}{4}\\
&=-\frac{1}{2}\left[\left\{\log\det(E_{2n}-A_n)+\trace(A_n)+\frac{1}{2}\trace(A_n^2)\right\}
+\trace((E_{2n}-A_n)^{-1}-E_{2n}-A_n-A_n^2)\right].
\end{align*}
Note that it holds that $(E_{2n}-A_n)^{-1}=\sum_{p=0}^\infty A_n^p$ for sufficiently large $n$ because $\|A_n\|_\mathrm{sp}\to0$ as $n\to\infty$ by Proposition \ref{prop:main}. Combining this fact with inequality (v) from Appendix II of \cite{Davies1973}, we obtain
\begin{align*}
|\widetilde{\mathbb{E}}_{n,0}[\xi_n]|\leq
\frac{1}{2}\|A_n\|_\mathrm{sp}\|A_n\|_F^2\left\{\frac{1}{3}\frac{1}{(1-\|A_n\|_\mathrm{sp})^3}+\frac{1}{1-\|A_n\|_\mathrm{sp}}\right\}
\end{align*}
for sufficiently large $n$. Hence Proposition \ref{prop:main} yields $\widetilde{\mathbb{E}}_{n,0}[\xi_n]\to0$.

Next, noting that $\xi_n$ can be rewritten as
\begin{align*}
\xi_n=-\frac{1}{2}\left(\widehat{z}_n^\top B_n\widehat{z}_n-\widetilde{E}_{n,0}[\widehat{z}_n^\top B_n\widehat{z}_n]\right)+\widetilde{E}_{n,0}[\xi_n],
\end{align*}
where $B_n=V_n(r_nu_n)^{-1}-V_n(0)^{-1}-\bar{G}_n^{-\frac{1}{2}}A_n\bar{G}_n^{-\frac{1}{2}},$ we obtain from Eq.(4.4) of \cite{Davies1973}
\begin{align*}
2\variance_{\widetilde{\mathbb{P}}_{n,0}}[\xi_n]&=\left\|V_n(0)^{\frac{1}{2}}(V_n(r_nu_n)^{-1}-V_n(0)^{-1})V_n(0)^{\frac{1}{2}}-A_n\right\|_F^2
=\left\|(E_{2n}-A_n)^{-1}-E_{2n}-A_n\right\|_F^2.
\end{align*}
Therefore, using the identity $(E_{2n}-A_n)^{-1}=\sum_{p=0}^\infty A_n^p$ again we obtain $2\variance_{\widetilde{\mathbb{P}}_{n,0}}[\xi_n]\leq\|A_n\|_\mathrm{sp}^2\|A_n\|_F^2(1-\|A_n\|_\mathrm{sp})^{-2}$ for sufficiently large $n$. Hence Proposition \ref{prop:main} again yields $\variance_{\widetilde{\mathbb{P}}_{n,0}}[\xi_n]\to0$, and we obtain \eqref{expansion2}.
\end{proof}

We finish this section with some remarks.
\begin{rmk}\label{intuition}
It is worth mentioning that we can infer from \cite{HRY2013} why the rate $n^{-\frac{3}{2}}$ is the proper convergence rate of our model in the case of $v_n\equiv0$ as follows. Let us set
\[
\mathcal{U}^n(\vartheta)=\sum_{i,j=1}^{n-1}(X_{i+1}-X_i)(Y_{j+1}-Y_j)1_{\{[\frac{i}{n},\frac{i+1}{n})\cap[\frac{j}{n}-\vartheta,\frac{j+1}{n}-\vartheta)\neq\emptyset\}}
\]
for $\vartheta\in\mathbb{R}$. The principle used in \cite{HRY2013} is that $\mathcal{U}^n(\vartheta)$ is close to the true correlation $\rho$ if and only if $\vartheta$ is close to the true time-lag parameter. Since the accuracy of estimating the correlation parameter is of order $1/\sqrt{n}$, we naturally consider the quantity $|\sqrt{n}(\mathcal{U}^n(\vartheta)-\rho)|$ to measure the distance between $\mathcal{U}^n(\vartheta)$ and $\rho$: $|\sqrt{n}(\mathcal{U}^n(\vartheta)-\rho)|$ would take a large value if $\vartheta$ is not sufficiently close to the true time-lag parameter. In fact, Proposition 1 from \cite{HRY2013} implies that $|\sqrt{n}(\mathcal{U}^n(\vartheta)-\rho)|$ does not diverge if and only if the distance between $\vartheta$ and the true time-lag parameter is at most of order $n^{-\frac{3}{2}}$. This information allows us to estimate the true time-lag parameter with the accuracy of order $n^{-\frac{3}{2}}$.
\end{rmk}

\begin{rmk}
From an econometric point of view, Proposition \ref{prop:hellinger} is of independent interest because the model given by \eqref{KLmodel1}--\eqref{KLmodel2} has an economic interpretation different from model \eqref{model}. This model contains measurement errors correlated to the latent returns $B_{i/n}-B_{(i-1)/n}$. 
The integrated volatility estimation in the presence of this type of measurement error has been studied by \citet{KL2008} for example. 
In the market microstructure theory, such a correlation is often explained as an effect of asymmetric information (e.g.~\citet{Glosten1987}). Interestingly, some economic arguments suggest that such an information asymmetry would cause a lead-lag effect; see \citet{Chan1993} and \citet{CSS2011} for instance.   
It would also be worth emphasizing that \citet{deJS2010} connect this type of model with the investigation of price discovery, for price discovery processes are closely related to lead-lag effects, as seen in \citet{deJMS1998} and \citet{Has1995}.  
\end{rmk}

\begin{rmk}
Our proof of the main result heavily depends on the Gaussianity of the model, and especially we require the Gaussianity of the measurement errors. It is obvious that we need some restriction on the distribution of the measurement errors to derive a specific limit experiment. In fact, if $v_n\equiv1$ and $\epsilon_i$'s take their values only in integers, then we can completely recover the signal for sufficiently large $n$. Apart from such a trivial example, the recent study of \cite{BJR2015} has shown that another (non-trivial) specification for the distribution of the measurement errors $\delta_i$'s in \eqref{GJmodel} can improve the convergence rate for estimating the scale parameter $\sigma$. In the light of the connection between the convergence rates for models \eqref{model}--\eqref{GJmodel}, we naturally conjecture that a similar specification for the measurement errors would affect the convergence rate for our model. This issue is beyond the scope of this paper and left to future research.  
\end{rmk}

\section{Efficient estimation of the lag parameter}\label{section:application}

As an application of the results from the previous section, we construct efficient estimators for the lag parameter $\vartheta$ in the models $(\mathbb{P}_{n,\vartheta})_{\vartheta\in\mathbb{R}}$ at $\vartheta=0$. Here we consider a slightly extended setup as follows: letting $\eta_n$ be a sequence of positive numbers tending to 0 and $\mathcal{C}$ be a bounded  open interval in $\mathbb{R}$, we construct efficient estimators for the parameter $c$ in the models $(\mathbb{P}_{n,c\eta_n})_{c\in \mathcal{C}}$ at every $c\in \mathcal{C}$. To make use of the results from the previous section, we impose the following condition on $\eta_n$:
\begin{equation}\label{eta}
\eta_n=o\left(n^{-1}\vee N_n^{-\frac{4}{3}}\right)\qquad
\text{and}\qquad
r_n^{-1}\eta_n\to\infty\qquad
\text{as }n\to\infty.
\end{equation}
Under \eqref{eta} there is a positive integer $n_0$ such that $c\eta_n\in\Theta_n$ and $V_n(c\eta_n)$ is invertible for any $c\in\mathcal{C}$ and $n\geq n_0$ due to Proposition \ref{prop:main}. Throughout this section, we always assume that $n$ is larger than such an $n_0$. 

\begin{rmk}
In a practical point of view, the dependence of the time-lag parameter $\vartheta$ on the sampling frequency $n$ is just a theoretical device to control the relative size of $\vartheta$ compared with $n$ (which corresponds to $\eta_n$) in the asymptotic theory and it is only important whether the asymptotic order condition (corresponding to \eqref{eta} in our case) is acceptable as an approximation. Namely, our asymptotic theory concerns whether the time-lag parameter $\vartheta$ is comparable to $n^{-\iota}$ for some $\iota>0$ given a fixed sampling frequency $n$ (the possible values of $\iota$ change in accordance with the noise level $v_n$) and it does not require that the time-lag parameter varies in proportion to the sampling frequency. This type of asymptotic theory is standard in econometrics: For example, when one considers the volatility estimation of a financial asset with taking account of rounding, one usually lets the rounding level shrink as the sampling frequency increases; see \cite{Rosenbaum2009}, \cite{LM2014}, \cite{LZL2015} and \cite{SK2015} for example.
\end{rmk}

We start with generalizing Proposition \ref{prop:main} by a matrix perturbation argument.
\begin{lem}\label{lemma:local}
For any $\alpha,\beta\in\mathbb{R}$ we have
\begin{align*}
&\sup_{c\in\mathcal{C}}\|V_n(c\eta_n)^{-\frac{1}{2}}(\alpha\bar{T}_n+\beta\bar{S}_n)V_n(c\eta_n)^{-\frac{1}{2}}\|_\mathrm{sp}=O(N_n),\\
&\sup_{c\in\mathcal{C}}\left|\rho^2r_n^2\|V_n(c\eta_n)^{-\frac{1}{2}}(\alpha\bar{T}_n+\beta\bar{S}_n)V_n(c\eta_n)^{-\frac{1}{2}}\|_F^2-2(\alpha^2I_\gamma+\beta^2J_\gamma)\right|\to0
\end{align*}
as $n\to\infty$.
\end{lem}

\begin{proof}
Setting $H_n(c)=V_n(c\eta_n)^{-\frac{1}{2}}\bar{G}_n^{\frac{1}{2}}$, we have
\[ 
V_n(c\eta_n)^{-\frac{1}{2}}(\alpha\bar{T}_n+\beta\bar{S}_n)V_n(c\eta_n)^{-\frac{1}{2}}=H_n(c)\bar{G}_n^{-\frac{1}{2}}(\alpha\bar{T}_n+\beta\bar{S}_n)\bar{G}_n^{-\frac{1}{2}}H_n(c)^\top
\]
for any $c\in\mathcal{C}$. 
Therefore, Ostrowski's theorem (Theorem 4.5.9 of \cite{HJ1985}) implies that 
\[
\|V_n(c\eta_n)^{-\frac{1}{2}}(\alpha\bar{T}_n+\beta\bar{S}_n)V_n(c\eta_n)^{-\frac{1}{2}}\|_\mathrm{sp}
\leq\|H_n(c)H_n(c)^\top\|_\mathrm{sp}\|\bar{G}_n^{-\frac{1}{2}}(\alpha\bar{T}_n+\beta\bar{S}_n)\bar{G}_n^{-\frac{1}{2}}\|_\mathrm{sp}
\]
and
\begin{align*}
&|\|V_n(c\eta_n)^{-\frac{1}{2}}(\alpha\bar{T}_n+\beta\bar{S}_n)V_n(c\eta_n)^{-\frac{1}{2}}\|_F^2-\|\bar{G}_n^{-\frac{1}{2}}(\alpha\bar{T}_n+\beta\bar{S}_n)\bar{G}_n^{-\frac{1}{2}}\|_F^2|\\
&\leq\|(H_n(c)H_n(c)^\top)^2-E_{2n}\|_\mathrm{sp}\|\bar{G}_n^{-\frac{1}{2}}(\alpha\bar{T}_n+\beta\bar{S}_n)\bar{G}_n^{-\frac{1}{2}}\|_F^2\\
&\leq\|H_n(c)H_n(c)^\top-E_{2n}\|_\mathrm{sp}(\|H_n(c)H_n(c)^\top\|_\mathrm{sp}+1)\|\bar{G}_n^{-\frac{1}{2}}(\alpha\bar{T}_n+\beta\bar{S}_n)\bar{G}_n^{-\frac{1}{2}}\|_F^2.
\end{align*}
Hence, Proposition \ref{prop:main} implies that the proof is completed once we show that $\sup_{c\in\mathcal{C}}\|H_n(c)H_n(c)^\top\|_\mathrm{sp}=O(1)$ and $\sup_{c\in\mathcal{C}}\|H_n(c)H_n(c)^\top-E_{2n}\|_\mathrm{sp}=o(1)$ as $n\to\infty$. Since $H_n(c)H_n(c)^\top$ and $H_n(c)^\top H_n(c)$ share the same eigenvalues (Theorem 1.3.20 of \cite{HJ1985}) and $H_n(c)^\top H_n(c)=(E_{2n}-\rho \eta_nG_n^{-\frac{1}{2}}(|c|\bar{S}_n+c\bar{T}_n)G_n^{-\frac{1}{2}})^{-1}$, the desired results follow from Proposition \ref{prop:main}, \eqref{eta} and the Neumann series representation of $H_n(c)^\top H_n(c)$. 
\end{proof}

Using the above result, we can prove a uniform version of Theorem \ref{thm:main}.
\begin{prop}\label{prop:infinitesimal}
Let $\mathsf{T}_n$ and $\mathsf{S}_n$ be defined by \eqref{ts}. Then 
\[
\log\frac{\mathrm{d}\mathbb{P}_{n,c\eta_n+r_nu_n}}{\mathrm{d}\mathbb{P}_{n,c\eta_n}}
-\left\{u_n\mathsf{T}_n+|u_n|\mathsf{S}_n-\frac{u_n^2}{2}(I_\gamma+J_\gamma)\right\}\xrightarrow{p}0
\]
under $\mathbb{P}_{n,c\eta_n}$ as $n\to\infty$ uniformly in $c\in\mathcal{C}$ for any bounded sequence $u_n$ of real numbers. Moreover, $(\mathsf{T}_n,\mathsf{S}_n)\xrightarrow{d}\mathcal{N}(0,I_\gamma)\otimes\mathcal{N}(0,J_\gamma)$ under $\mathbb{P}_{n,c\eta_n}$ as $n\to\infty$ uniformly in $c\in\mathcal{C}$. 
\end{prop}

\begin{proof}
We can prove the first claim in a similar manner to the proof of Theorem \ref{thm:main} using Lemma \ref{lemma:local} instead of Proposition \ref{prop:main}. To prove the second claim, it suffices to show that $(\mathsf{T}_n,\mathsf{S}_n)\xrightarrow{d}  \mathcal{N}(0,I_\gamma)\otimes   \mathcal{N}(0,J_\gamma)$ under $\mathbb{P}_{n,c_n\eta_n}$ as $n\to\infty$ for any sequence $c_n$ of numbers in $ \mathcal{C}$, which follows from Lemma \ref{lemma:local} and inequality (13) from \cite{DY2011}.
\end{proof}

Proposition \ref{prop:infinitesimal} implies that the experiments $(\mathbb{P}_{n,c\eta_n})_{c\in \mathcal{C}}$ enjoy the LAN property if $\gamma=\infty$ and do not otherwise. 
When the LAN property holds true, there is a well-established theory to define the asymptotic efficiency of estimators (cf.~Section II-11 of \cite{IH1981}): A sequence $c_n$ of estimators in the experiments $(\mathbb{P}_{n,c\eta_n})_{c\in \mathcal{C}}$ is said to be asymptotically efficient at $c\in\mathcal{C}$ if the variables $r_n^{-1}\eta_n(c_n-c)$ converge in law to $\mathcal{N}(0,I_\gamma^{-1})$ under $\mathbb{P}_{n,c\eta_n}$ as $n\to\infty$ (see Definition II-11.1 of \cite{IH1981}). Under the LAN property, this definition of the asymptotic efficiency is supported by several theorems such as the convolution theorem (e.g.~Theorem II-9.1 of \cite{IH1981}) and the local asymptotic minimax theorem (e.g.~Theorem II-12.1 of \cite{IH1981}). Moreover, it is well-known that both maximum likelihood and Bayesian estimators are asymptotically efficient under very general settings (cf.~Chapter III of \cite{IH1981}). On the other hand, if the LAN property fails, it is generally not obvious how to define the asymptotic efficiency of estimators. Here, we adopt the approach from \citet{KK2000} to define the asymptotic efficiency, which is based on Theorem I-9.1 of \cite{IH1981} that derives an asymptotically minimax lower bound from the asymptotic properties of the Bayesian estimators. As a consequence, the Bayesian estimators are turned out to be asymptotically efficient. 

Now we explain the strategy to obtain asymptotically efficient estimators in our setting. As in the previous section, we would like to work with the tractable model $\widetilde{\mathbb{P}}_{n,\vartheta}$ rather than the original model $\mathbb{P}_{n,\vartheta}$. For this reason we consider the quasi-likelihood function based on the former as follows:
\[
\mathbb{L}_n(c):=\frac{\mathrm{d}\widetilde{\mathbb{P}}_{n,c\eta_n}}{\mathrm{d}x}=\frac{1}{(2\pi)^{n}\sqrt{\det V_n(c\eta_n)}}\exp\left(-\frac{1}{2}\widehat{z}_n^\top V_n(c\eta_n)^{-1}\widehat{z}_n\right),\qquad c\in\mathcal{C}.
\]
Then we consider the quasi maximum likelihood and Bayesian estimators based on $\mathbb{L}_n(c)$ as our estimators and give their asymptotic behavior in the experiments $(\widetilde{\mathbb{P}}_{n,c\eta_n})_{c\in \mathcal{C}}$ using the general scheme of \citet{IH1981} (see Proposition \ref{prop:IH}).  
Next we consider the case $\gamma=\infty$ where the LAN property holds true and thus convergence in law in $(\widetilde{\mathbb{P}}_{n,c\eta_n})_{c\in \mathcal{C}}$ can be transferred to that in $(\mathbb{P}_{n,c\eta_n})_{c\in \mathcal{C}}$ by Proposition \ref{prop:hellinger}(c). Finally, we consider the case $\gamma<\infty$ where we have $(\widetilde{\mathbb{P}}_{n,c\eta_n})_{c\in \mathcal{C}}=(\mathbb{P}_{n,c\eta_n})_{c\in \mathcal{C}}$ for sufficiently large $n$ due to \eqref{eta} and Proposition \ref{prop:hellinger}(a), hence we can apply the Ibragimov-Has'minskii method to define and obtain asymptotically efficient estimators.

The quasi maximum likelihood estimator (QMLE) $\hat{c}_n$ is defined as a solution of the equation
\[
\mathbb{L}_n(\hat{c}_n)=\sup_{c\in\mathcal{C}}\mathbb{L}_n(c).
\]
Note that the above equation always has at least one solution belonging to the closure of $\mathcal{C}$ because $c\mapsto\mathbb{L}_n(c)$ is continuous. 
Moreover, we can choose $\hat{c}_n$ so that it is measurable by the measurable selection theorem (see e.g.~Theorem 6.7.22 of \cite{Pfanzagl1994}).  
Also, the quasi Bayesian estimator (QBE) $\tilde{c}_n$ for a prior density $q: \mathcal{C}\to(0,\infty)$ with respect to the quadratic loss is defined by
\[
\tilde{c}_n=\int_\mathcal{C}c\mathbb{L}_n(c)q(c)\mathrm{d}c\left/\int_\mathcal{C}\mathbb{L}_n(c)q(c)\mathrm{d}c,\right.
\]
where the prior density $q$ is assumed to be continuous and satisfy $0<\inf_{c\in \mathcal{C}}q(c)\leq\sup_{c\in \mathcal{C}}q(c)<\infty$. The corresponding QMLE and QBE in the experiments $(\mathbb{P}_{n,\vartheta})_{\vartheta\in\mathbb{R}}$ are given by $\hat{\vartheta}_n=\hat{c}_n\eta_n$ and $\tilde{\vartheta}_n=\tilde{c}_n\eta_n$, respectively. 
\begin{rmk}
Since the quantity $\eta_n$ seems the \textit{exact} order of the true time-lag parameter, one may consider that in a practical setting it is difficult to know $\eta_n$ beforehand and thus it is difficult to use the estimators $\hat{\vartheta}_n$ and $\tilde{\vartheta}_n$. However, when we construct the estimator $\hat{\vartheta}_n$, $\eta_n$ can be considered as the \textit{maximum} order of the true time-lag parameter as follows. Let us set $\mathbb{L}_n'(\vartheta)=\mathrm{d}\widetilde{\mathbb{P}}_{n,\vartheta}/\mathrm{d}x$ and $\mathcal{C}_n=\{c\eta_n:c\in\mathcal{C}\}$. Then the estimator $\hat{\vartheta}_n$ can be considered as a solution of the equation
\[
\mathbb{L}'_n(\hat{\vartheta}_n)=\sup_{\vartheta\in\mathcal{C}_n}\mathbb{L}_n(\vartheta).
\]
Therefore, in a practical situation $(\sup_{c\in\mathcal{C}}c)\eta_n$ (resp.~$(\inf_{c\in\mathcal{C}}c)\eta_n$) can be interpreted as an upper bound (resp.~a lower bound) of possible time-lag parameters $\vartheta$. It is often not so difficult to find such bounds in a practical setting (and they are typically ``small'' as pointed out in the Introduction). For example, we can find them by computing the cross-correlations via \cite{HRY2013}'s method as in \cite{HA2014}. The same remark can be applied to the estimator $\tilde{\vartheta}_n$ because it can be rewritten as 
\[
\tilde{\vartheta}_n=\int_{\mathcal{C}_n}\vartheta\mathbb{L}'_n(\vartheta)q_n(\vartheta)\mathrm{d}\vartheta\left/\int_{\mathcal{C}_n}\mathbb{L}'_n(\vartheta)q_n(\vartheta)\mathrm{d}\vartheta,\right.
\]
where $q_n(\vartheta)=\eta_n^{-1}q(\vartheta/\eta_n)$ for $\vartheta\in\mathcal{C}_n$ (so $q_n$ is a prior density on $\mathcal{C}_n$).
\end{rmk}
To describe the limit distribution of these estimators, we introduce the likelihood ratio process for the limit experiment
\[
Z(u)=\exp\left(u \zeta_1+|u| \zeta_2-\frac{u^2}{2}(I_\gamma+J_\gamma)\right),\qquad
u\in\mathbb{R},
\]
where $ \zeta_1$ and $ \zeta_2$ are two mutually independent variables such that $ \zeta_1\sim\mathcal{N}(0,I_\gamma)$ and $ \zeta_2\sim\mathcal{N}(0,J_\gamma)$. Then we set
\[
\hat{u}=\argmax_{u\in\mathbb{R}}Z(u)
=\left\{
\begin{array}{ll}
( \zeta_1+ \zeta_2)/(I_\gamma+J_\gamma)  & \textrm{if } \zeta_1\geq(- \zeta_2)\vee0,  \\
( \zeta_1- \zeta_2)/(I_\gamma+J_\gamma)  &  \textrm{if } \zeta_1< \zeta_2\wedge0,    \\
0  &    \textrm{otherwise}  
\end{array}
\right.
\]
and
\begin{align*}
\tilde{u}=\frac{\int_{-\infty}^\infty uZ(u)\mathrm{d}u}{\int_{-\infty}^\infty Z(u)\mathrm{d}u}.
\end{align*}

We first give the asymptotic behavior of the estimators $\hat{c}_n$ and $\tilde{c}_n$ in the experiments $(\widetilde{\mathbb{P}}_{n,c\eta_n})_{c\in \mathcal{C}}$ using the general scheme of \citet{IH1981}. Note that in this situation $\hat{c}_n$ and $\tilde{c}_n$ are true maximum likelihood and Bayesian estimators, respectively.
\begin{prop}\label{prop:IH}
For any compact subset $\mathcal{K}$ of $\mathcal{C}$, uniformly in $c\in\mathcal{K}$ it holds that $r_n^{-1}\eta_n(\hat{c}_n-c)$ converges in law to $\hat{u}$ under $\widetilde{\mathbb{P}}_{n,c\eta_n}$ and $\widetilde{\mathbb{E}}_{n,c\eta_n}[|r_n^{-1}\eta_n(\hat{c}_n-c)|^p]\to E[|\hat{u}|^p]$ for any $p>0$ as $n\to\infty$. Also, uniformly in $c\in\mathcal{K}$ it holds that $r_n^{-1}\eta_n(\tilde{c}_n-c)$ converges in law to $\tilde{u}$ under $\widetilde{\mathbb{P}}_{n,c\eta_n}$ and $\widetilde{\mathbb{E}}_{n,c\eta_n}[|r_n^{-1}\eta_n(\tilde{c}_n-c)|^p]\to E[|\tilde{u}|^p]$ for any $p>0$ as $n\to\infty$.
\end{prop}

\begin{proof}
For every $c\in \mathcal{C}$, we set $\mathbb{U}_n(c)=\{u\in\mathbb{R}:c+r_n\eta_n^{-1}u\in \mathcal{C}\}$ and define $Z_{n,c}(u)=\mathrm{d}\widetilde{\mathbb{P}}_{n,c\eta_n+r_nu}/\mathrm{d}\widetilde{\mathbb{P}}_{n,c\eta_n}$ for each $u\in\mathbb{U}_n(c)$. According to Theorems I-10.1 and I-10.2 from \cite{IH1981}, it suffices to prove the following statements:
\begin{enumerate}[label=(\alph*)]

\item $\limsup_{n\to\infty}\sup_{c\in\mathcal{K}}\sup_{u,v\in\mathbb{U}_n(c)}|u-v|^{-2}\widetilde{\mathbb{E}}_{n,c\eta_n}[|\sqrt{ Z_{n,c}(u)}-\sqrt{ Z_{n,c}(v)}|^2]<\infty$, \label{IH1}

\item there is a constant $\kappa>0$ such that $\limsup_{n\to\infty}\sup_{c\in\mathcal{K}}\sup_{u\in\mathbb{U}_n(c)}e^{\kappa u^2}\widetilde{\mathbb{E}}_{n,c\eta_n}[\sqrt{ Z_{n,c}(u)}]<\infty$, \label{IH2}

\item the marginal distributions of $ Z_{n,c}$ converge in law to the marginal distributions of $ Z$ uniformly in $c\in\mathcal{K}$ as $n\to\infty$. \label{IH3}

\end{enumerate}

\ref{IH3} is an immediate consequence of Proposition \ref{prop:infinitesimal}. On the other hand, by Eq.(A.4) from \cite{Reiss2011} we obtain 
\begin{align*}
&\widetilde{\mathbb{E}}_{n,c\eta_n}\left[\left|\sqrt{ Z_{n,c}(u)}-\sqrt{ Z_{n,c}(v)}\right|^2\right]
=H^2(\widetilde{\mathbb{P}}_{n,c\eta_n+r_nu},\widetilde{\mathbb{P}}_{n,c\eta_n+r_nv})\\
&\leq4\rho^2r_n^2(u-v)^2\left(\|V_n(c\eta_n+r_nu)^{-\frac{1}{2}}\bar{T}_nV_n(c\eta_n+r_nu)^{-\frac{1}{2}}\|_F^2+\|V_n(c\eta_n+r_nu)^{-\frac{1}{2}}\bar{S}_nV_n(c\eta_n+r_nu)^{-\frac{1}{2}}\|_F^2\right),
\end{align*}
hence Lemma \ref{lemma:local} yields claim \ref{IH1}.

Now we consider \ref{IH2}. By Corollary 3.2a.1 from \cite{MP1992} we have 
\begin{align*}
\log\widetilde{\mathbb{E}}_{n,c\eta_n}\left[\sqrt{ Z_{n,c}(u)}\right]
=-\frac{1}{4}\log\det[E_{2n}-A_{n,c}(u)]
-\frac{1}{2}\log\det\left(E_{2n}+\frac{1}{2}\left((E_{2n}-A_{n,c}(u))^{-1}-E_{2n}\right)\right),
\end{align*}
where $A_{n,c}(u)=E_{2n}-V_n(c\eta_n)^{-\frac{1}{2}}V_n(c\eta_n+r_nu)V_n(c\eta_n)^{-\frac{1}{2}}$. Then we consider the following decomposition:
\begin{align*}
&\log\widetilde{\mathbb{E}}_{n,c\eta_n}\left[\sqrt{ Z_{n,c}(u)}\right]\\
&=-\frac{1}{4}\left\{\log\det[E_{2n}-A_{n,c}(u)]+\trace [A_{n,c}(u)]+\frac{1}{2}\|A_{n,c}(u)\|_F^2\right\}\\
&\quad-\frac{1}{2}\left\{\log\det\left(E_{2n}+\frac{1}{2}B_{n,c}(u)\right)-\frac{1}{2}\trace\left(B_{n,c}(u)\right)+\frac{1}{8}\|B_{n,c}(u)\|_F^2\right\}\\
&\quad+\frac{1}{4}\left\{\trace [A_{n,c}(u)]+\|A_{n,c}(u)\|_F^2-\trace\left(B_{n,c}(u)\right)\right\}-\frac{1}{8}\left\{\|A_{n,c}(u)\|_F^2-\frac{1}{2}\|B_{n,c}(u)\|_F^2\right\}\\
&=:\mathbb{I}_{n,c}(u)+\mathbb{II}_{n,c}(u)+\mathbb{III}_{n,c}(u)+\mathbb{IV}_{n,c}(u),
\end{align*}
where $B_{n,c}(u)=(E_{2n}-A_{n,c}(u))^{-1}-E_{2n}$. Let us set 
\[
\tilde{T}_{n,c}=\rho V_n(c\eta_n)^{-\frac{1}{2}}\bar{T}_nV_n(c\eta_n)^{-\frac{1}{2}},\qquad\tilde{S}_{n,c}=\rho V_n(c\eta_n)^{-\frac{1}{2}}\bar{S}_nV_n(c\eta_n)^{-\frac{1}{2}}.
\] 
Then we have $A_{n,c}(u)=c\eta_n\tilde{T}_{n,c}+|c\eta_n|\tilde{S}_{n,c}-(c\eta_n+r_nu)\tilde{T}_{n,c}-|c\eta_n+r_nu|\tilde{S}_{n,c}$, hence it holds that 
\[
\sup_{c\in \mathcal{C}}\sup_{u\in\mathbb{U}_n(c)}\|A_{n,c}(u)\|_\mathrm{sp}\leq\sup_{c\in\mathcal{C}}|c|\sup_{c\in\mathcal{C}}2\eta_n(\|\tilde{T}_{n,c}\|_\mathrm{sp}+\|\tilde{S}_{n,c}\|_\mathrm{sp}).
\]
Here, we use the fact that $c+r_n\eta_n^{-1}u\in\mathcal{C}$ because of $u\in\mathbb{U}_n(c)$. In particular, we have $\sup_{c\in \mathcal{C}}\sup_{u\in\mathbb{U}_n(c)}\|A_{n,c}(u)\|_\mathrm{sp}<1$ for sufficiently large $n$ by \eqref{eta} and Lemma \ref{lemma:local}. For such an $n$ we have $B_{n,c}(u)=\sum_{k=1}^\infty A_{n,c}(u)^k$ and thus
\[
\|B_{n,c}(u)\|_\mathrm{sp}\leq\frac{\|A_{n,c}(u)\|_\mathrm{sp}}{1-\|A_{n,c}(u)\|_\mathrm{sp}},\qquad
\|B_{n,c}(u)\|_F\leq\frac{\|A_{n,c}(u)\|_F}{1-\|A_{n,c}(u)\|_\mathrm{sp}},
\]
where we use the inequality $\|A_{n,c}(u)^k\|_F\leq\|A_{n,c}(u)\|_F\|A_{n,c}(u)\|_\mathrm{sp}^{k-1}$ for $k\geq1$ to obtain the latter estimate. Therefore, for sufficiently large $n$ we have for any $c\in \mathcal{C}$ and any $u\in\mathbb{U}_n(c)$
\begin{align*}
\left|\mathbb{I}_{n,c}(u)\right|
\leq\frac{1}{12}\frac{\|A_{n,c}(u)\|_\mathrm{sp}\|A_{n,c}(u)\|_F^2}{1-\|A_{n,c}(u)\|_\mathrm{sp}^3},\qquad
\left|\mathbb{II}_{n,c}(u)\right|
\leq\frac{1}{6}\frac{\|B_{n,c}(u)\|_\mathrm{sp}\|B_{n,c}(u)\|_F^2}{8-\|B_{n,c}(u)\|_\mathrm{sp}^3}
\end{align*}
by Appendix II-(v) from \cite{Davies1973} and
\begin{align*}
\left|\mathbb{III}_{n,c}(u)\right|
\leq\sum_{k=3}^\infty\left|\trace(A_{n,c}(u)^k)\right|
\leq\|A_{n,c}(u)\|_F^2\frac{\|A_{n,c}(u)\|_\mathrm{sp}}{1-\|A_{n,c}(u)\|_\mathrm{sp}}
\end{align*}
by the inequality $\left|\trace(A_{n,c}(u)^k)\right|\leq\|A_{n,c}(u)\|_F^2\|A_{n,c}(u)\|_\mathrm{sp}^{k-2}$ for $k\geq3$, as well as
\begin{align*}
\mathbb{IV}_{n,c}(u)
\leq-\frac{\|A_{n,c}(u)\|_F^2}{8}\left(1-\frac{1}{2(1-\|A_{n,c}(u)\|_\mathrm{sp})^2}\right).
\end{align*}
Consequently, there is a constant $\kappa_0>0$ such that for sufficiently large $n$ it holds that
\[
\log\widetilde{\mathbb{E}}_{n,c\eta_n}\left[\sqrt{ Z_{n,c}(u)}\right]
\leq-\kappa_0\|A_{n,c}(u)\|_F^2
\]
for any $c\in \mathcal{C}$ and any $u\in\mathbb{U}_n(c)$. Now we consider giving an upper bound for $-\|A_{n,c}(u)\|_F^2$. We have
\begin{align*}
-\|A_{n,c}(u)\|_F^2
&=-u^2r_n^2\|\tilde{T}_{n,c}\|_F^2-2r_nu(|c\eta_n+r_nu|-|c\eta_n|)\trace(\tilde{T}_{n,c}\tilde{S}_{n,c})-(|c\eta_n+r_nu|-|c\eta_n|)^2\|\tilde{S}_{n,c}\|_F^2\\
&\leq-2u^2I_\gamma+u^2\left(\left|r_n^2\|\tilde{T}_{n,c}\|_F^2-2I_\gamma\right|+2\left|r_n^2\trace(\tilde{T}_{n,c}\tilde{S}_{n,c})\right|\right).
\end{align*}
Therefore, noting $I_\gamma>0$, by Lemma \ref{lemma:local} for sufficiently large $n$ we have
$-\|A_{n,c}(u)\|_F^2
\leq-I_\gamma u^2$ 
for any $c\in \mathcal{C}$ and any $u\in\mathbb{U}_n(c)$. Consequently, we obtain \ref{IH2} by setting $\kappa=\kappa_0I_\gamma$.
\end{proof}

If $\gamma<\infty$, \eqref{eta} is equivalent to the condition that $\eta_n=o(n^{-1})$ as $n\to\infty$. Therefore, Proposition \ref{prop:hellinger}(a) yields the following result:
\begin{cor}\label{coincide}
If $\gamma<\infty$, the statement of Proposition \ref{prop:IH} still holds true while $\widetilde{\mathbb{P}}_{n,c\eta_n}$'s are replaced by $\mathbb{P}_{n,c\eta_n}$'s. 
\end{cor}

Now we return to the efficient estimation of the parameter $c$ in the model $(\mathbb{P}_{n,c\eta_n})_{c\in \mathcal{C}}$. First we consider the case $\gamma=\infty$. In this case we know that $(\mathbb{P}_{n,c\eta_n})_{c\in \mathcal{C}}$ enjoy the LAN property at every $c\in \mathcal{C}$ by Proposition \ref{prop:infinitesimal}, so the definition of the asymptotic efficiency of an estimator sequence is well-established as explained in the above.  

\begin{theorem}
If $\gamma=\infty$, both $\hat{c}_n$ and $\tilde{c}_n$ are asymptotically efficient at every $c\in\mathcal{C}$ in the experiments $(\mathbb{P}_{n,c\eta_n})_{c\in \mathcal{C}}$. That is, both $r_n^{-1}\eta_n(\hat{c}_n-c)$ and $r_n^{-1}\eta_n(\tilde{c}_n-c)$ converge in law to $  \mathcal{N}(0,I_\gamma^{-1})$ under $\mathbb{P}_{n,c\eta_n}$ for any $c\in \mathcal{C}$ as $n\to\infty$. In particular, both $\hat{\vartheta}_n$ and $\tilde{\vartheta}_n$ are asymptotically efficient at $\vartheta=0$ in the experiments $(\mathbb{P}_{n,\vartheta})_{\vartheta\in \mathbb{R}}$. 
\end{theorem}

Next we turn to the case $\gamma<\infty$. In this case the experiments $(\mathbb{P}_{n,c\eta_n})_{c\in \mathcal{C}}$ no longer enjoy the LAN property, so the definition of the asymptotic efficiency is not obvious. As explained in the above, here we follow the approach of \citet{KK2000} to define the asymptotic efficiency for our experiments. We obtain the following result by virtue of Corollary \ref{coincide} and Theorem I-9.1 of \cite{IH1981}:
\begin{theorem}\label{thm:minimax}
If $\gamma<\infty$, we have
\[
\lim_{\delta\to0}\liminf_{n\to\infty}\sup_{|c-c_0|<\delta}(r_n^{-1}\eta_n)^2\mathbb{E}_{n,c\eta_n}[(c^*_n-c)^2]\geq E[\tilde{u}^2]
\]
for any $c_0\in \mathcal{C}$ and any estimator sequence $c^*_n$ in the experiments $(\mathbb{P}_{n,c\eta_n})_{c\in \mathcal{C}}$. In particular, we also have
\[
\lim_{\delta\to0}\liminf_{n\to\infty}\sup_{|\vartheta|<\delta\eta_n}r_n^{-2}\mathbb{E}_{n,\vartheta}[(\vartheta^*_n-\vartheta)^2]\geq E[\tilde{u}^2]
\]
for any estimator sequence $\vartheta^*_n$ in the experiments $(\mathbb{P}_{n,\vartheta})_{\vartheta\in \mathbb{R}}$.
\end{theorem}

Thanks to Theorem \ref{thm:minimax}, an estimator sequence $c^*_n$ is said to be asymptotically efficient at $c_0\in \mathcal{C}$ in the experiments $(\mathbb{P}_{n,c\eta_n})_{c\in \mathcal{C}}$ if it holds that
\[
\lim_{\delta\to0}\liminf_{n\to\infty}\sup_{|c-c_0|<\delta}(r_n^{-1}\eta_n)^2\mathbb{E}_{n,c\eta_n}[(c^*_n-c)^2]=E[\tilde{u}^2].
\]
Similarly, an estimator sequence $\vartheta^*_n$ is said to be asymptotically efficient at $\vartheta=0$ in the experiments $(\mathbb{P}_{n,\vartheta})_{\vartheta\in \mathbb{R}}$ if it holds that
\[
\lim_{\delta\to0}\liminf_{n\to\infty}\sup_{|\vartheta|<\delta\eta_n}r_n^{-2}\mathbb{E}_{n,\vartheta}[(\vartheta^*_n-\vartheta)^2]=E[\tilde{u}^2]
\]
for any sequence $\eta_n$ of positive numbers satisfying \eqref{eta}. 
The following result is an immediate consequence of this definition. 
\begin{theorem}
If $\gamma<\infty$, the sequence $\tilde{c}_n$ is asymptotically efficient at every $c\in\mathcal{C}$ in the experiments $(\mathbb{P}_{n,c\eta_n})_{c\in \mathcal{C}}$. In particular, the sequence $\tilde{\vartheta}_n$ is asymptotically efficient at $\vartheta=0$ in the experiments $(\mathbb{P}_{n,\vartheta})_{\vartheta\in\mathbb{R}}$.
\end{theorem}

In contrast, there is no guarantee of the asymptotic efficiency of the (Q)MLE $\hat{c}_n$ if $\gamma<\infty$. In fact, $\tilde{c}_n$ may perform much better than $\hat{c}_n$ if $\gamma=0$, as shown in the following proposition.
\begin{prop}
It holds that
\begin{align}
E\left[\hat{u}^2\right]
&=\frac{1}{I_\gamma+J_\gamma}\left(1-\frac{1}{\pi}\arctan\left(\sqrt{\frac{J_\gamma}{I_\gamma}}\right)+\frac{\sqrt{I_\gamma J_\gamma}}{\pi(I_\gamma+J_\gamma)}\right),\label{mle:var}\\
E\left[\tilde{u}^2\right]
&=\frac{1}{I_\gamma+J_\gamma}\int_{-\infty}^\infty\int_{-\infty}^\infty\left(\frac{x\Psi(x)-y\Psi(y)}{\Psi(x)+\Psi(y)}\right)^2\psi_R(x,y)\mathrm{d}x\mathrm{d}y,\label{bayes:var}
\end{align}
where $\Psi(x)=\int_0^\infty e^{ux-u^2/2}\mathrm{d}u$ and $\psi_R(x,y)$ denotes the bivariate normal density with standard normal marginals and correlation $R=(J_\gamma-I_\gamma)/(J_\gamma+I_\gamma)$. In particular, $\lim_{|\rho|\to1}E\left[\tilde{u}^2\right]/E\left[\hat{u}^2\right]=0$ if $\gamma=0$.
\end{prop}

\begin{proof}
Let us denote by $\phi_a$ the normal density with mean 0 and variance $a$. Then, a simple calculation yields
\begin{align*}
E\left[\hat{u}^2\right]
&=\frac{2}{(I_\gamma+J_\gamma)^2}\int_0^\infty z^2\mathrm{d}z\int_{0}^\infty \phi_{I_\gamma}(x)\phi_{J_\gamma}(z-x)\mathrm{d}x.
\end{align*}
By formulae (3.322.2) and (6.292) from \cite{GR2007} we have
\begin{align*}
\int_0^\infty z^2\mathrm{d}z\int_{0}^\infty \phi_{I_\gamma}(x)\phi_{J_\gamma}(z-x)\mathrm{d}x
=\frac{I_\gamma+J_\gamma}{2}-\frac{I_\gamma+J_\gamma}{2\pi}\arctan\left(\sqrt{\frac{J_\gamma}{I_\gamma}}\right)+\frac{\sqrt{I_\gamma J_\gamma}}{2\pi},
\end{align*}
hence we obtain \eqref{mle:var}.

Next, by a change of variable we obtain
\[
\int_{-\infty}^\infty Z(u)\mathrm{d}u=\frac{1}{\sqrt{I_\gamma+J_\gamma}}\left\{\Psi\left(\frac{ \zeta_1+ \zeta_2}{\sqrt{I_\gamma+J_\gamma}}\right)+\Psi\left(\frac{ \zeta_2- \zeta_1}{\sqrt{I_\gamma+J_\gamma}}\right)\right\}.
\]
Moreover, formulae (3.462.5) and (3.322.2) from \cite{GR2007} imply that
\[
\int_{-\infty}^\infty uZ(u)\mathrm{d}u=\frac{1}{I_\gamma+J_\gamma}\left\{\frac{ \zeta_1+ \zeta_2}{\sqrt{I_\gamma+J_\gamma}}\Psi\left(\frac{ \zeta_1+ \zeta_2}{\sqrt{I_\gamma+J_\gamma}}\right)-\frac{ \zeta_2- \zeta_1}{\sqrt{I_\gamma+J_\gamma}}\Psi\left(\frac{ \zeta_2- \zeta_1}{\sqrt{I_\gamma+J_\gamma}}\right)\right\}.
\]
Since the distribution of the vector $(( \zeta_1+ \zeta_2)/\sqrt{I_\gamma+J_\gamma},( \zeta_2- \zeta_1)/\sqrt{I_\gamma+J_\gamma})$ has the density $\psi_R(x,y)$, we obtain \eqref{bayes:var}.

Finally, we prove the latter statement. Define the functions $f$ and $g$ on $(-1,1)$ by
\[
f(r)=1-\frac{1}{\pi}\arctan\left(\sqrt{\frac{1+r}{1-r}}\right)+\frac{\sqrt{1-r^2}}{2\pi},\qquad
g(r)=\int_{-\infty}^\infty\int_{-\infty}^\infty\left(\frac{x\Psi(x)-y\Psi(y)}{\Psi(x)+\Psi(y)}\right)^2\psi_{r}(x,y)\mathrm{d}x\mathrm{d}y.
\]
Then we have $E[\hat{u}^2]=f(R)/(I_\gamma+J_\gamma)$ and $E[\tilde{u}^2]=g(R)/(I_\gamma+J_\gamma)$. Since $R\to1$ as $|\rho|\to1$ if $\gamma=0$ and $\lim_{r\to1}f(r)=\frac{1}{2}$, it suffices to prove $\lim_{r\to1}g(r)=0$. Because we have
\[
g(r)=\int_{-\infty}^\infty\int_{-\infty}^\infty\left(\frac{x\Psi(x)-(rx+\sqrt{1-r^2}y)\Psi(rx+\sqrt{1-r^2}y)}{\Psi(x)+\Psi(rx+\sqrt{1-r^2}y)}\right)^2\phi_{1}(x)\phi_1(y)\mathrm{d}x\mathrm{d}y,
\]
the dominated convergence theorem yields $\lim_{r\to1}g(r)=0$, which completes the proof.
\end{proof}

\section{Appendix: Proof of Proposition \ref{prop:main}}\label{section:technical}

Before starting the proof, we introduce some notation. We set
\[
\xi_i:=\xi_{i,n}:=\frac{\pi}{2n+1}\left(i-\frac{1}{2}\right),\qquad i=1,2,\dots.
\]
Then we define the $n\times n$ matrix $U_n=(u_{ij})_{1\leq i,j\leq n}$ by
\[ 
u_{ij}:=\frac{2}{\sqrt{2n+1}}\cos\left[\frac{2\pi}{2n+1}\left(i-\frac{1}{2}\right)\left(j-\frac{1}{2}\right)\right]=\frac{2}{\sqrt{2n+1}}\cos\left[\xi_i(2j-1)\right],
\]
which is often referred to as the Discrete Cosine Transform (DCT) of type-VIII (see \cite{SSH2014} and references therein). Note that $U_n^\top=U_n$ and $U_n$ is real orthogonal. It is known that $U_n$ diagonalizes $F_n$ as follows:
\begin{equation}\label{eq:dct}
U_nF_n U_n
=\diag(\lambda_1,\dots,\lambda_n),\qquad
\text{where }
\lambda_i:=2\left[1-\cos\left(2\xi_i\right)\right].
\end{equation}
See Lemma 1 of \cite{KS2013} or Lemma C.2 of \cite{SSH2014} for the proof. 

For each $a>0$ we define the functions $f_a$ and $g_a$ on $\mathbb{R}$ by
\[
f_a(x)=\frac{a}{n}+2v_n(1-\cos(x)),\qquad
g_a(x)=\frac{\sin(x)}{f_a(x)}\qquad
(x\in\mathbb{R}).
\]
We also set $G_n(a)=\frac{a}{n}E_n+v_n F_n$. From \eqref{eq:dct} we have
\begin{equation}\label{eq:dct2}
U_nG_n(a)U_n=\Lambda_n(a):=\diag(f_a(2\xi_1),\dots,f_a(2\xi_n)).
\end{equation}

\begin{rmk}\label{fisher}
It turns out that the off-diagonal components of $U_nT_nU_n$ play a dominant role to calculate the limit of $\|\bar{G}_n^{-\frac{1}{2}}\bar{T}_n\bar{G}_n^{-\frac{1}{2}}\|_F$. This is essentially different from the case of calculating the Fisher information for the scale parameter estimation from observations of the form \eqref{GJmodel}, where the similarity transformations of Toeplitz matrices by $U_n$ are sufficiently approximated by diagonal matrices as manifested by Lemma C.4 of \cite{SSH2014}. For this reason we need rather specific calculations as seen in Lemmas \ref{lemma:BS} and \ref{lemma:lemmaT}. 
\end{rmk}

For a square matrix $A$, $\spr(A)$ denotes the spectral radius of $A$. We will frequently use the identity $\|A\|_\mathrm{sp}=\spr(A)$ holding if $A$ is a normal matrix.

Now we start the main body of the proof. We will frequently use the following well-known inequality for the sine function,
\begin{equation}\label{jordan}
\frac{2}{\pi}x\leq\sin(x)\leq1\qquad\left(0\leq x\leq\frac{\pi}{2}\right).
\end{equation}


\begin{lem}\label{lemma:calculus}
For any $a>0$, we have
\[
\sup_{0\leq x\leq\pi}g_a(x)=\frac{1}{\sqrt{(\frac{a}{n})^2+4\frac{a}{n}v_n}},\qquad
\sup_{0\leq x\leq\pi}|g'_a(x)|\leq\frac{3n}{a}.
\]
\end{lem}

\begin{proof}
The claim immediately follows from the identity 
$g'_a(x)=\frac{(\frac{a}{n}+2v_n)\cos(x)-2v_n}{f_a(x)^2}=\frac{\frac{a}{n}(1+\cos(x))}{f_a(x)^2}-\frac{1}{f_a(x)}.$
\end{proof}

\begin{lem}\label{smallnoise}
Let $\psi:[0,\pi]\to\mathbb{R}$ be continuous. Also, let $m_n$ be a sequence of positive integers such that $m_n\leq n$ and $m_n/n\to c\in(0,1]$ as $n\to\infty$. Then, for any $p>0$ we have 
\begin{align*}
\lim_{n\to\infty}\frac{1}{n^{p+1}}\sum_{i=1}^{m_n}\frac{\psi(2\xi_i)}{f_a(2\xi_i)^p}
=\frac{1}{a^p\pi}\int_0^{\pi c} \psi(x)\mathrm{d}x
\end{align*}
provided that $\gamma=0$.
\end{lem}

\begin{proof}
By the fundamental theorem of calculus, we have
\begin{align*}
\left|\frac{\psi(2\xi_i)}{f_a(2\xi_i)^p}-\frac{\psi(2\xi_i)}{(a/n)^p}\right|
\leq p\|\psi\|_\infty\left|\int_0^{v_n\lambda_i}\frac{1}{(a/n+x)^{p+1}}\mathrm{d}x\right|
\leq\frac{2p\|\psi\|_\infty v_n}{(a/n)^{p+1}}.
\end{align*}
Hence the desired result follows from the standard Riemann sum approximation.
\end{proof}


\begin{lem}\label{lambda.inv}
Let $m_n$ be a sequence of positive integers such that $m_n\leq n$ and $m_n/n\to c\in(0,1]$ as $n\to\infty$. Then
\begin{align}
&\lim_{n\to\infty}\frac{1}{nN_n}\sum_{i=1}^{m_n}\frac{1}{f_a(2\xi_i)}=
\left\{\begin{array}{ll}
\frac{2}{\pi\sqrt{a(a+4\gamma)}}\arctan\left(\sqrt{1+\frac{4\gamma}{a}}\tan\left(\frac{\pi c}{2}\right)\right)&\mathrm{if}~\gamma<\infty,\\
\frac{1}{2\sqrt{a}}&\mathrm{if}~\gamma=\infty,\\
\end{array}\right.\label{eq:lambda.inv1}\\
&\lim_{n\to\infty}r_n^2\sum_{i=1}^{m_n}g_{a}(2\xi_i)g_{b}(2\xi_i)=
\left\{\begin{array}{ll}
\frac{1}{2ab}\left(c-\frac{\sin 2\pi c}{2\pi}\right)&\mathrm{if}~\gamma=0,\\
\frac{\sqrt{b(b+4\gamma)}}{2\pi\gamma^2(b-a)}\arctan\left(\sqrt{1+\frac{4\gamma}{b}}\tan\left(\frac{\pi c}{2}\right)\right)&\\
-\frac{\sqrt{a(a+4\gamma)}}{2\pi\gamma^2(b-a)}\arctan\left(\sqrt{1+\frac{4\gamma}{a}}\tan\left(\frac{\pi c}{2}\right)\right)-\frac{c}{4\gamma^2}&\mathrm{if}~0<\gamma<\infty,\\
\frac{1}{2(\sqrt{a}+\sqrt{b})}&\mathrm{if}~\gamma=\infty.
\end{array}\right.\label{eq:lambda.inv2}
\end{align}
for any $a,b>0$ such that $a\neq b$.
\end{lem}

\begin{proof}
First, using the lower and upper Darboux sums of the integral $\int_{2\xi_1}^{2\xi_{m_n}}\frac{1}{f_a(x)}\mathrm{d}x$, we obtain
\begin{align*}
\frac{2n+1}{2\pi}\int_{2\xi_1}^{2\xi_{m_n}}\frac{1}{f_a(x)}\mathrm{d}x+\frac{1}{f_a(2\xi_{m_n})}
\leq\sum_{i=1}^{m_n}\frac{1}{f_a(2\xi_i)}
\leq\frac{1}{f_a(2\xi_1)}+\frac{2n+1}{2\pi}\int_{2\xi_1}^{2\xi_{m_n}}\frac{1}{f_a(x)}\mathrm{d}x.
\end{align*}
Now, formula (2.553.3) in \cite{GR2007} yields 
\[
\int_0^y\frac{1}{f_a(x)}\mathrm{d}x
=\frac{2n}{\sqrt{a(a+4nv_n)}}\arctan\left(\sqrt{1+\frac{4nv_n}{a}}\tan\left(\frac{y}{2}\right)\right),
\]
hence we obtain \eqref{eq:lambda.inv1}. Next, a simple calculation yields
\begin{align*}
g_{a}(x)g_{b}(x)
&=\frac{n}{b-a}\left\{\frac{\sin^2 x}{f_a(x)}-\frac{\sin^2 x}{f_b(x)}\right\}.
\end{align*}
Therefore, if $\gamma=0$, Lemma \ref{smallnoise} implies that
\begin{align*}
\lim_{n\to\infty}r_n^2\sum_{i=1}^{m_n}g_{a}(2\xi_i)g_{b}(2\xi_i)
=\frac{1}{\pi ab}\int_0^{\pi c}\sin^2x\mathrm{d}x
=\frac{1}{2ab}\left(c-\frac{\sin 2\pi c}{2\pi}\right).
\end{align*}
On the other hand, if $\gamma\neq0$, for sufficiently large $n$ we have $1-\cos(x)=(f_a(x)-a/n)/(2v_n)$, hence
\begin{align*}
\sin^2x
&=f_a(x)/v_n-a/(nv_n)-f_a(x)^2/(4v_n^2)+f_a(x)a/(2nv_n^2)-a^2/(2nv_n)^2.
\end{align*}
Therefore, we obtain
\begin{align*}
\frac{\sin^2 x}{f_a(x)}-\frac{\sin^2 x}{f_b(x)}
&=\left\{\frac{b}{nv_n}+\frac{b^2}{(2nv_n)^2}\right\}\frac{1}{f_b(x)}-\left\{\frac{a}{nv_n}+\frac{a^2}{(2nv_n)^2}\right\}\frac{1}{f_a(x)}-\frac{b-a}{4nv_n^2}.
\end{align*}
Hence the desired result follows from \eqref{eq:lambda.inv1}.
\end{proof}


\begin{lem}\label{lemma:BS}
Let $m_n$ be a sequence of positive integers such that $m_n\leq n$ and $m_n\to\infty$ as $n\to\infty$. Then
\[
\lim_{n\to\infty}\left(\frac{4}{2n+1}\right)^2\sum_{i=1}^{m_n}\frac{\sin^4(n\cdot 2\xi_1i)}{\sin^2(2\xi_1i)}=2.
\]
\end{lem}

\begin{proof}
Since $4n\xi_1=\pi -2\xi_1$, we have
\begin{align*}
\sin^2\left(n\cdot 2\xi_1l\right)=\frac{1}{2}\left\{1-\cos\left(4n\xi_1l\right)\right\}
=\frac{1}{2}\left\{1-(-1)^l\cos\left(2\xi_1l\right)\right\}
=\left\{
\begin{array}{ll}
\sin^2\left(\xi_1l\right)&\textrm{if $l$ is even},\\
\cos^2\left(\xi_1l\right)&\textrm{if $l$ is odd}.
\end{array}\right.
\end{align*}
Therefore, using the formula $\sin(2\xi_1i)=2\sin(\xi_1 i)\cos(\xi_1 i)$, we can decompose the target quantity as
\begin{align*}
\left(\frac{4}{2n+1}\right)^2\sum_{i=1}^{m_n}\frac{\sin^4(n\cdot 2\xi_1i)}{\sin^2(2\xi_1i)}
=\frac{4}{(2n+1)^2}\left\{\sum_{\begin{subarray}{c}
i=1\\
i:\textrm{ even}
\end{subarray}}^{m_n}\tan^2(\xi_1i)
+\sum_{\begin{subarray}{c}
i=1\\
i:\textrm{ odd}
\end{subarray}}^{m_n}\cot^2(\xi_1i)
\right\}
=:\mathbf{A}_{1,n}+\mathbf{A}_{2,n}.
\end{align*}

First we prove $\lim_n\mathbf{A}_{1,n}=0$. Using the monotonicity of the tangent function and assumption $m_n\leq n$, we obtain
\begin{align*}
\mathbf{A}_{1,n}\leq\frac{8}{(2n+1)\pi}\int_0^{\frac{\pi}{2}\frac{n+1}{2n+1}}\tan^2(x)\mathrm{d}x
\end{align*}
Since formula (2.526.22) in \cite{GR2007} yields $\lim_n\int_0^{\frac{\pi}{2}\frac{n+1}{2n+1}}\tan^2(x)\mathrm{d}x=1-\pi/4$, we conclude that $\lim_n\mathbf{A}_{1,n}=0$.

Next we prove $\lim_n\mathbf{A}_{2,n}=2$. Our proof relies on the following inequality for the tangent function:
\begin{equation}\label{BSinequality}
\frac{\pi}{2}x\leq\tan\left(\frac{\pi}{2}x\right)\leq\frac{\pi}{2}\frac{x}{1-x^2}\qquad
(0\leq x<1).
\end{equation}
The lower estimate of \eqref{BSinequality} is well-known, and the upper estimate is known as the Becker-Stark inequality (Eq.(2) of \cite{BS1978}). Now, using \eqref{BSinequality}, we obtain
\begin{align*}
\frac{16}{\pi^2}\sum_{i=1}^{\lfloor \frac{m_n+1}{2}\rfloor}\left\{\frac{1}{(2i-1)^2}-\frac{2}{(2n+1)^2}+\frac{(2i-1)^2}{(2n+1)^4}\right\}\leq 
\mathbf{A}_{2,n}\leq\frac{16}{\pi^2}\sum_{i=1}^{\lfloor \frac{m_n+1}{2}\rfloor}\frac{1}{(2i-1)^2}.
\end{align*}
Therefore, using formula $\sum_{i=1}^\infty(2i-1)^{-2}=\frac{\pi^2}{8}$, we conclude that $\lim_n\mathbf{A}_{2,n}=2$.
\end{proof}

\begin{lem}\label{lemma:normR}
For any $a,b>0$, it holds that $\|G_n(a)^{-\frac{1}{2}}R_nG_n(b)^{-\frac{1}{2}}\|_\mathrm{sp}=O(N_n)$ and $\|G_n(a)^{-\frac{1}{2}}R_nG_n(b)^{-\frac{1}{2}}\|_F^2=O(N_n^{2})$ as $n\to\infty$.
\end{lem}

\begin{proof}
First, by definition we have $\|G_n(a)^{-\frac{1}{2}}R_nG_n(b)^{-\frac{1}{2}}\|_\mathrm{sp}=\spr(G_n(b)^{-\frac{1}{2}}G_n(a)^{-\frac{1}{2}}R_n)$, hence \eqref{eq:dct2} and Theorem 5.6.9 of \cite{HJ1985} yield
\begin{align*}
\|G_n(a)^{-\frac{1}{2}}R_nG_n(b)^{-\frac{1}{2}}\|_\mathrm{sp}&\leq\max_{1\leq i\leq n}\left|\sum_{k=1}^n\frac{u_{ik}u_{k1}}{\sqrt{f_a(2\xi_k)f_b(2\xi_k)}}\right|
\leq\frac{4}{2n+1}\sum_{k=1}^n\left\{\frac{1}{f_a(2\xi_k)}+\frac{1}{f_b(2\xi_k)}\right\}.
\end{align*}
Therefore, Lemma \ref{lambda.inv} implies that $\|G_n(a)^{-\frac{1}{2}}R_nG_n(b)^{-\frac{1}{2}}\|_\mathrm{sp}=O(N_n)$. Moreover, since it holds that
\begin{align*}
\|G_n(a)^{-\frac{1}{2}}R_nG_n(b)^{-\frac{1}{2}}\|_F^2&=
\trace(G_n(a)^{-1}R_nG_n(b)^{-1}R_n)\\
&=\sum_{k,l=1}^n\frac{u_{1k}u_{k1}u_{1l}u_{l1}}{f_a(2\xi_k)f_b(2\xi_l)}
\leq\left(\frac{4}{2n+1}\sum_{k=1}^n\frac{1}{f_a(2\xi_k)}\right)\left(\frac{4}{2n+1}\sum_{l=1}^n\frac{1}{f_b(2\xi_l)}\right),
\end{align*}
Lemma \ref{lambda.inv} again yields the desired result.
\end{proof}

\begin{lem}\label{lemma:lemmaS}
For any $a>0$, we have
\begin{equation}\label{eq:lemmaS}
\|G_n(a)^{-\frac{1}{2}}S_nG_n(a)^{-\frac{1}{2}}\|^2_\mathrm{sp}=O(N_n),\qquad
r_n^2\|G_n(a)^{-\frac{1}{2}}S_nG_n(a)^{-\frac{1}{2}}\|^2_F\to J^0_\gamma(a)
\end{equation}
as $n\to\infty$, where for any $a>0$ we set
\begin{align*}
J^0_\gamma(a)
=\left\{\begin{array}{ll}
\frac{6}{a^2}&\mathrm{if}~\gamma=0,\\
\frac{1}{2\gamma^2}\left(2-3\left(\frac{a}{a+4\gamma}\right)^{1/2}+\left(\frac{a}{a+4\gamma}\right)^{3/2}\right)&\mathrm{if}~0<\gamma<\infty,\\
0 &\mathrm{if}~\gamma=\infty.
\end{array}\right.
\end{align*}
\end{lem}

\begin{proof}
Since $S_n-R_n=F_n$, by Lemma \ref{lemma:normR} it suffices to prove \eqref{eq:lemmaS} in the case where $S_n$ is replaced by $F_n$. \eqref{eq:dct}--\eqref{eq:dct2} imply that 
\begin{equation}\label{eq:lemmaS1}
\|G_n(a)^{-\frac{1}{2}}F_nG_n(a)^{-\frac{1}{2}}\|_\mathrm{sp}=\max_{1\leq i\leq n}\frac{\lambda_i}{f_a(2\xi_i)}\leq\min\{v_n^{-1},4n/a\}
\end{equation}
and
\begin{equation}\label{eq:lemmaS2}
r_n^2\|G_n(a)^{-\frac{1}{2}}F_nG_n(a)^{-\frac{1}{2}}\|^2_F
=r_n^2\sum_{i=1}^n\frac{\lambda_i^2}{f_a(2\xi_i)^2}.
\end{equation}
The first equation in \eqref{eq:lemmaS} immediately follows from \eqref{eq:lemmaS1}. In order to prove the second equation in \eqref{eq:lemmaS}, we will prove the right side of \eqref{eq:lemmaS2} converges to $J^0_\gamma(a)$ as $n\to\infty$.

First, if $\gamma=0$, the desired result follows from Lemma \ref{smallnoise}.

Next, if $0<\gamma<\infty$, noting that $f_a(2\xi_i)=a/n+v_n\lambda_i$, we have
\begin{align*}
\frac{\lambda_i^2}{f_a(2\xi_i)^2}=\frac{1}{v_n^2}\left\{1-2\frac{a}{n}\frac{1}{f_a(2\xi_i)}+\left(\frac{a}{n}\right)^2\frac{1}{f_a(2\xi_i)^2}\right\},
\end{align*}
hence by Lemma \ref{lambda.inv} we obtain the desired equation once we prove
\begin{align*}
\lim_{n\to\infty}\frac{1}{n^3}\sum_{i=1}^{n}\frac{1}{f_a(2\xi_i)^2}
=\frac{1}{2a\sqrt{a^2+4a\gamma}}\left(1+\frac{a}{a+4\gamma}\right).
\end{align*}
The monotonicity of the cosine function yields
\begin{align*}
\frac{2n+1}{2\pi}\int_{0}^\pi\frac{1}{f_a(x)^2}\mathrm{d}x
\leq\sum_{i=1}^{n}\frac{1}{f_a(2\xi_i)^2}
\leq\frac{1}{f_a(2\xi_1)^2}+\frac{2n+1}{2\pi}\int_{0}^{\pi}\frac{1}{f_a(x)^2}\mathrm{d}x.
\end{align*}
Since formula (3.661.4) in \cite{GR2007} implies that 
\[
\int_0^\pi\frac{1}{f_a(x)^2}\mathrm{d}x
=\frac{\pi}{2(a/n)\sqrt{(a/n)^2+4av_n/n}}\left(1+\frac{a/n}{a/n+4v_n}\right),
\]
we obtain the desired result.

Finally, if $\gamma=\infty$, using the inequality $f_a(2\xi_i)\geq v_n\lambda_i$ we obtain
\begin{align*}
r_n^2\sum_{i=1}^{n}\frac{\lambda_i^2}{f_a(2\xi_i)^2}
\leq\frac{n}{N_n^3v_n^2}=\frac{1}{\sqrt{nv_n}},
\end{align*}
hence we deduce the desired result.
\end{proof}

\begin{lem}\label{lemma:lemmaT}
If $a$ and $b$ are positive numbers such that $a\neq b$, we have
\[
\lim_{n\to\infty}r_n^2\|G_n(a)^{-\frac{1}{2}}T_nG_n(b)^{-\frac{1}{2}}\|^2_F
=\left\{
\begin{array}{ll}
\frac{2}{ab}  & \mathrm{if}~\gamma=0,  \\
\frac{\sqrt{b(b+4\gamma)}-\sqrt{a(a+4\gamma)}}{\gamma^2(b-a)}-\frac{1}{\gamma^2}&\mathrm{if}~0<\gamma<\infty,\\
\frac{2}{\sqrt{a}+\sqrt{b}}  & \mathrm{if}~\gamma=\infty.   
\end{array}
\right.
\]
\end{lem}

\begin{proof}
Since $u_{0i}=u_{1i}$, we have
\[
(U_nT_nU_n)^{ij}+(U_nR_nU_n)^{ij}=(U_n(T_n+R_n)U_n)^{ij}=\sum_{k=1}^n(u_{i,k-1}-u_{i,k+1})u_{kj}.
\]
Using the trigonometric identities $\cos(x)-\cos(y)=-2\sin(\frac{x+y}{2})\sin(\frac{x-y}{2})$ and $2\sin(x)\cos(y)=\sin(x+y)-\sin(x-y)$, we obtain
\[
(U_nT_nU_n)^{ij}+(U_nR_nU_n)^{ij}=\frac{4}{2n+1}\sin(2\xi_i)\sum_{k=1}^n\left\{\sin\left(\left(2k-1\right)(\xi_i+\xi_j)\right)+\sin\left(\left(2k-1\right)(\xi_i-\xi_j)\right)\right\}.
\]
Then, using summation formula (1.342.3) of \cite{GR2007}, we have
\begin{equation}\label{eq:summation}
(U_nT_nU_n)^{ij}+(U_nR_nU_n)^{ij}=\frac{4}{2n+1}\sin(2\xi_i)\left\{\frac{\sin^2(n(\xi_i+\xi_j))}{\sin(\xi_i+\xi_j)}+\frac{\sin^2(n(\xi_i-\xi_j))}{\sin(\xi_i-\xi_j)}1_{\{i\neq j\}}\right\}.
\end{equation}
Now, since $(U_nT_nU_n)^\top=-U_nT_nU_n$ and $(U_nR_nU_n)^\top=U_nR_nU_n$, by \eqref{eq:dct2}, \eqref{eq:summation} and the unitary invariance of the Frobenius norm, we obtain
\begin{align*}
&\|G_n(a)^{-\frac{1}{2}}T_nG_n(b)^{-\frac{1}{2}}\|_F^2-\|G_n(a)^{-\frac{1}{2}}R_nG_n(b)^{-\frac{1}{2}}\|_F^2\\
&=-\left(\frac{4}{2n+1}\right)^2\sum_{i,j=1}^ng_a(2\xi_i)g_b(2\xi_j)\left\{\frac{\sin^4(n(\xi_i+\xi_j))}{\sin^2(\xi_i+\xi_j)}-\frac{\sin^4(n(\xi_i-\xi_j))}{\sin^2(\xi_i-\xi_j)}1_{\{i<j\}}-\frac{\sin^4(n(\xi_i-\xi_j))}{\sin^2(\xi_i-\xi_j)}1_{\{i>j\}}\right\}\\
&=: \mathbf{B}_{1,n}+ \mathbf{B}_{2,n}+ \mathbf{B}_{3,n}.
\end{align*}
First we consider $ \mathbf{B}_{1,n}$. Using inequalities \eqref{jordan} and $(x+y)^2\geq4xy$ ($x,y\in\mathbb{R}$), we have
\begin{align*}
| \mathbf{B}_{1,n}|
&\leq4\left(\max_{1\leq i\leq n}g_a(2\xi_i)\right)\left(\max_{1\leq i\leq n}g_b(2\xi_i)\right)\sum_{i,j=1}^n\frac{1}{(i-\frac{1}{2}+j-\frac{1}{2})^2}\\
&\leq4\left(\max_{1\leq i\leq n}g_a(2\xi_i)\right)\left(\max_{1\leq i\leq n}g_b(2\xi_i)\right)\left(\sum_{i=1}^n\frac{1}{2i-1}\right)^2,
\end{align*}
and thus Lemma \ref{lemma:calculus} yields $ \mathbf{B}_{1,n}=O((N_n\log n)^2)=o(r_n^{-2})$.

Next we consider $ \mathbf{B}_{2,n}$. First we prove $ \mathbf{B}_{2,n}= \mathbf{B}_{2,n}'+o(r_n^{-2})$, where
\[
 \mathbf{B}_{2,n}'=\left(\frac{4}{2n+1}\right)^2\sum_{i,j=1}^ng_{a}(2\xi_i)g_{b}(2\xi_i)\frac{\sin^4(n(\xi_i-\xi_j))}{\sin^2(\xi_i-\xi_j)}1_{\{i<j\}}.
\]
Lemma \ref{lemma:calculus} and \eqref{jordan} yield
\begin{align*}
\left| \mathbf{B}_{2,n}- \mathbf{B}_{2,n}'\right|
&\leq\left(\frac{4}{2n+1}\right)^2\frac{6n}{b}\sum_{i,j=1}^ng_{a}(2\xi_i)|\xi_i-\xi_j|\frac{\sin^4(n(\xi_i-\xi_j))}{\sin^2(\xi_i-\xi_j)}1_{\{i<j\}}\\
&\leq\frac{4}{(2n+1)^2}\frac{6\pi^2n}{b}\sum_{i,j=1}^ng_{a}(2\xi_i)\frac{1}{|\xi_i-\xi_j|}1_{\{i<j\}}\\
&\leq\frac{4}{2n+1}\frac{6\pi n}{b}\sum_{i=1}^ng_{a}(2\xi_i)\sum_{j=1}^n\frac{1}{j}.
\end{align*}
If $\gamma=\infty$, by the property of $g_a$ we have
\begin{align*}
\frac{2\pi}{2n+1}\sum_{i=1}^ng_{a}(2\xi_i)
&\leq\frac{2\pi}{2n+1}2\max_{1\leq i\leq n}g_a(2\xi_i)+\int_{2\xi_1}^{2\xi_n} g_{a}(x)\mathrm{d}x\\
&\leq\frac{2\pi}{2n+1}2\max_{1\leq i\leq n}g_a(2\xi_i)+\frac{1}{2v_n}\log\frac{f_a(2\xi_n)}{f_a(2\xi_1)},
\end{align*}
hence Lemma \ref{lemma:calculus} yields $\sum_{i=1}^ng_{a}(2\xi_i)=O(N_n^2\log n)$. This also holds true in the case that $\gamma<\infty$ due to Lemma \ref{lambda.inv} because $g_a(x)\leq f_a(x)^{-1}$. Consequently, $\left| \mathbf{B}_{2,n}- \mathbf{B}_{2,n}'\right|=O((N_n\log n)^2)=o(r_n^{-2})$. Now, since $\xi_i-\xi_j=2\xi_1(i-j)$, for any $c\in(0,1)$ we have
\begin{align*}
\left(\frac{4}{2n+1}\right)^2\sum_{i=1}^{\lfloor nc\rfloor} g_{a}(2\xi_i)g_{b}(2\xi_i)\sum_{j=1}^{n-\lfloor nc\rfloor}\frac{\sin^4(n\cdot 2j\xi_1)}{\sin^2(2j\xi_1)}
\leq
 \mathbf{B}_{2,n}'
\leq\left(\frac{4}{2n+1}\right)^2\sum_{i=1}^n g_{a}(2\xi_i)g_{b}(2\xi_i)\sum_{j=1}^n\frac{\sin^4(n\cdot 2j\xi_1)}{\sin^2(2j\xi_1)}.
\end{align*}
Therefore, Lemma \ref{lemma:BS} implies that
\begin{align*}
2\liminf_{n\to\infty}r_n^2\sum_{i=1}^{\lfloor nc\rfloor} g_{a}(2\xi_i)g_{b}(2\xi_i)
\leq
\liminf_{n\to\infty}r_n^2 \mathbf{B}_{2,n}'
\leq
\limsup_{n\to\infty}r_n^2 \mathbf{B}_{2,n}'
\leq2\limsup_{n\to\infty}r_n^2\sum_{i=1}^n g_{a}(2\xi_i)g_{b}(2\xi_i).
\end{align*}
Then, letting $c\uparrow1$, by Lemma \ref{lambda.inv} we obtain $\lim_{n\to\infty}r_n^2 \mathbf{B}_{2,n}=\lim_{n\to\infty}r_n^2 \mathbf{B}_{2,n}'=2\lim_{n\to\infty}r_n^2\sum_{i=1}^n g_{a}(2\xi_i)g_{b}(2\xi_i)$. 

By symmetry we have $\lim_{n\to\infty}r_n^2 \mathbf{B}_{3,n}=\lim_{n\to\infty}r_n^2 \mathbf{B}_{2,n}$, hence we complete the proof due to \eqref{eq:lambda.inv2} and Lemma \ref{lemma:normR}.
\end{proof}

\begin{proof}[\textbf{\upshape{Proof of Proposition \ref{prop:main}}}]
Set
\begin{align*}
\bar{U}_n
=\frac{1}{\sqrt{2}}\left[
\begin{array}{cc}
U_n & -U_n \\
U_n & U_n
\end{array}
\right].
\end{align*}
Then we have
\begin{align*}
\bar{G}_n^{-\frac{1}{2}}\bar{S}_n\bar{G}_n^{-\frac{1}{2}}
&=\frac{1}{2}\bar{U}_n\left[
\begin{array}{cc}
\overline{L}_nU_n(S_n+\rho^{-1}R_n)U_n\overline{L}_n & 0 \\
0 & -\underline{L}_nU_n(S_n-\rho^{-1}R_n)U_n\underline{L}_n
\end{array}
\right]\bar{U}_n^\top
\end{align*}
and
\begin{align*}
\bar{G}_n^{-\frac{1}{2}}\bar{T}_n\bar{G}_n^{-\frac{1}{2}}
&=\frac{1}{2}\bar{U}_n\left[
\begin{array}{cc}
0 & \overline{L}_nU_n(T_n+\rho^{-1}R_n)U_n\underline{L}_n \\
-\underline{L}_nU_n(T_n-\rho^{-1}R_n)U_n\overline{L}_n & 0
\end{array}
\right]\bar{U}_n^\top,
\end{align*}
where $\overline{L}_n=\Lambda_n(1+\rho)^{-\frac{1}{2}}$ and $\underline{L}_n=\Lambda_n(1-\rho)^{-\frac{1}{2}}$. 
Hence we obtain
\begin{align*}
&4\|\bar{G}_n^{-\frac{1}{2}}(\alpha\bar{T}_n+\beta\bar{S}_n)\bar{G}_n^{-\frac{1}{2}}\|_F^2\\
&=2\alpha^2\|G_n(1+\rho)^{-\frac{1}{2}}(T_n+\rho^{-1}R_n)G_n(1-\rho)^{-\frac{1}{2}}\|_F^2\\
&\quad+\beta^2(\|G_n(1+\rho)^{-\frac{1}{2}}(S_n+\rho^{-1}R_n)G_n(1+\rho)^{-\frac{1}{2}}\|_F^2
+\|G_n(1-\rho)^{-\frac{1}{2}}(S_n-\rho^{-1}R_n)G_n(1-\rho)^{-\frac{1}{2}}\|_F^2).
\end{align*}
Therefore, Lemmas \ref{lemma:normR}--\ref{lemma:lemmaT} yield \eqref{information}. On the other hand, since $2\|\bar{G}_n^{-\frac{1}{2}}\bar{S}_n\bar{G}_n^{-\frac{1}{2}}\|_\mathrm{sp}\leq\|G_n(1+\rho)^{-\frac{1}{2}}(S_n+\rho^{-1}R_n)G_n(1+\rho)^{-\frac{1}{2}}\|_\mathrm{sp}+\|G_n(1-\rho)^{-\frac{1}{2}}(S_n-\rho^{-1}R_n)G_n(1-\rho)^{-\frac{1}{2}}\|_\mathrm{sp}$, Lemmas \ref{lemma:normR}--\ref{lemma:lemmaS} also yield $\|\bar{G}_n^{-\frac{1}{2}}\bar{S}_n\bar{G}_n^{-\frac{1}{2}}\|_\mathrm{sp}=O(N_n)$. 

Hence the proof is completed once we prove $\|\bar{G}_n^{-\frac{1}{2}}\bar{T}_n\bar{G}_n^{-\frac{1}{2}}\|_\mathrm{sp}=O(N_n)$. Note that $\bar{G}_n^{-\frac{1}{2}}V_n(\vartheta)\bar{G}_n^{-\frac{1}{2}}$ is positive semidefinite and $\rho\vartheta\bar{G}_n^{-\frac{1}{2}}\bar{T}_n\bar{G}_n^{-\frac{1}{2}}+\bar{G}_n^{-\frac{1}{2}}V_n(\vartheta)\bar{G}_n^{-\frac{1}{2}}=E_{2n}-\rho|\vartheta|\bar{G}_n^{-\frac{1}{2}}\bar{S}_n\bar{G}_n^{-\frac{1}{2}}$ for any $\vartheta\in\Theta_n$. Note also that both $\bar{T}_n$ and $\bar{S}_n$ are symmetric. Therefore, if $\lambda$ is an eigenvalue of $\bar{G}_n^{-\frac{1}{2}}\bar{T}_n\bar{G}_n^{-\frac{1}{2}}$, by the monotonicity theorem for eigenvalues (Corollary 4.3.3 of \cite{HJ1985}) we have $\rho\vartheta\lambda\leq\|E_{2n}-\rho|\vartheta|\bar{G}_n^{-\frac{1}{2}}\bar{S}_n\bar{G}_n^{-\frac{1}{2}}\|_\mathrm{sp}$ for any $\vartheta\in\Theta_n$. Since we can take $\vartheta=\pm N_n^{-1}$, this inequality implies that $|\rho|N_n^{-1}\|\bar{G}_n^{-\frac{1}{2}}\bar{T}_n\bar{G}_n^{-\frac{1}{2}}\|_\mathrm{sp}\leq\|E_{2n}-\rho N_n^{-1}\bar{G}_n^{-\frac{1}{2}}\bar{S}_n\bar{G}_n^{-\frac{1}{2}}\|_\mathrm{sp}\leq1+|\rho|N_n^{-1}\|\bar{G}_n^{-\frac{1}{2}}\bar{S}_n\bar{G}_n^{-\frac{1}{2}}\|_\mathrm{sp}$, which yields the desired result.
\end{proof}

\section*{Acknowledgements}

The author is grateful to two anonymous referees for their careful reading and insightful comments which have significantly improved a former version of this paper.  
The author also thanks the participants at ASC2013 Asymptotic Statistics and Computations, Statistics for Stochastic Processes and Analysis of High Frequency Data, Statistique Asymptotique des Processus Stochastiques X, and Statistics for Stochastic Processes and Analysis of High Frequency Data IV for valuable comments.  
This work was supported by CREST, JST.

{\small
\renewcommand*{\baselinestretch}{1}\selectfont
\addcontentsline{toc}{section}{References}

}

\end{document}